\documentclass[11pt,reqno]{amsart}
\usepackage{xifthen,setspace}
\usepackage{bbm}
\usepackage{esvect}
\usepackage{amsmath} 
\usepackage{lipsum}

\usepackage{pkgfile}

\renewcommand{\leq}{\leqslant}
\renewcommand{\geq}{\geqslant}

\newcommand{\vp}{{\phi}}
\newcommand{\bw}{{\beta_{\mathrm{CW}}^*(p)}}

\newcommand{\sa}{{\mathcal{C}}}
\newcommand{\p}{{\mathbb{P}}}
\newcommand{\e}{{\mathbb{E}}}
\newcommand{\bs}{{\bm X}}

\newcommand{\cs}{{\bar{X}_N}}
\newcommand{\bt}{{\bm x}}

\allowdisplaybreaks

\title[Estimation in Tensor Ising Models]{Estimation in Tensor Ising Models}

\author[Mukherjee]{Somabha Mukherjee} 
\address{Department of Statistics, University of Pennsylvania, Philadelphia, USA, {\tt somabha@wharton.upenn.edu}}

\author[Son]{Jaesung Son}
\address{Department of Statistics, Columbia University, New York, USA, {\tt  js4638@columbia.edu }}

\author[Bhattacharya]{Bhaswar B. Bhattacharya}
\address{Department of Statistics, University of Pennsylvania, Philadelphia, USA, {\tt bhaswar@wharton.upenn.edu}}

\begin{document}

\begin{abstract} 
The $p$-tensor Ising model is a one-parameter discrete exponential family for modeling dependent binary data, where the sufficient statistic is a multi-linear form of degree $p \geq 2$. This is a natural generalization of the matrix Ising model, that provides a convenient mathematical framework for capturing, not just pairwise, but higher-order dependencies in complex relational data. In this paper, we consider the problem of estimating the natural parameter of the $p$-tensor Ising model given a single sample from the distribution on $N$ nodes. Our estimate is based on the maximum pseudo-likelihood (MPL) method, which provides a  computationally efficient algorithm for estimating the parameter that avoids computing the intractable partition function. We derive general conditions under which the  MPL estimate is $\sqrt N$-consistent, that is, it converges to the true parameter at rate $1/\sqrt N$. Our conditions are robust enough to handle a variety of commonly used tensor Ising models, including spin glass models with random interactions and models where the rate of estimation undergoes a   phase  transition. In particular, this includes results on $\sqrt N$-consistency of the MPL estimate in the well-known $p$-spin Sherrington-Kirkpatrick (SK) model, spin systems on general $p$-uniform hypergraphs, and Ising models on the hypergraph stochastic block model (HSBM). In fact, for the HSBM we pin down the exact location of the phase transition threshold, which is determined by the positivity of a certain mean-field variational problem, such that above this threshold the MPL estimate is $\sqrt N$-consistent, while below the threshold no estimator is consistent. Finally, we derive the precise fluctuations of the MPL estimate in the special case of the $p$-tensor Curie-Weiss model, which is the Ising model on the complete $p$-uniform hypergraph. An interesting consequence of our results is that the MPL estimate in the Curie-Weiss model saturates the Cramer-Rao lower bound at all points above the estimation threshold, that is, the MPL estimate incurs no loss in asymptotic statistical efficiency in the estimability regime, even though it is obtained by minimizing only an approximation of the true likelihood function for computational tractability. 
\end{abstract}



	\maketitle

\section{Introduction}

The Ising model is a discrete exponential family with binary outcomes, where the sufficient statistic involves a quadratic term designed to capture correlations arising from pairwise interactions.  This was originally developed in statistical physics to model ferromagnetism \cite{ising}, and has since then found applications in various diverse fields, such as spatial statistics, social networks, computer vision, neural networks, and computational biology \cite{spatial,cd_trees,geman_graffinge,disease,neural,innovations}.  The increasing popularity of the Ising model as a foundational tool for understanding nearest-neighbor interactions in network data, has made it imperative to develop computationally tractable algorithms for learning the model parameters and understanding their rates of convergence (statistically efficiencies).  In particular, we are interested in estimating the parameters of the model given a single sample of binary outcomes from an underlying network. This problem was classically studied when the  underlying network was a spatial lattice, where consistency and optimality of the maximum likelihood (ML) estimates were derived~\cite{comets_exp,gidas,guyon,discrete_mrf_pickard}. However, for general networks, parameter estimation using the  ML method turns out to be notoriously hard due to the appearance of an intractable normalizing constant in the likelihood. To circumvent this issue, Chatterjee \cite{chatterjee} proposed using the maximum pseudolikelihood (MPL) estimator \cite{besag_lattice,besag_nl}, which is a computationally efficient algorithm for estimating the parameters of a  Markov random field, that maximizes an approximation to the likelihood function (a `pseudo-likelihood') based on conditional distributions. This method and results in \cite{chatterjee} were later generalized in \cite{BM16} and \cite{pg_sm} to obtain rates of estimation for Ising models on general weighted graphs and joint estimation of parameters, respectively. These techniques were recently  used in Daskalakis et al. \cite{cd_ising_II,cd_ising_I} to obtain rates of convergence of the MPLE in general logistic regression models with dependent observations. Very recently, Dagan et al. \cite{cd_ising_estimation} considered the problem of parameter estimation in a more general model where the binary outcomes can be influenced by various underlying networks, and, as a consequence, improved some of the results in \cite{BM16}. Related problems in hypothesis testing given a single sample from an Ising model are considered in \cite{gb_testing,ising_testing,rm_sm}.

In many situations, both in modeling real-world network data and  interacting spin systems, dependencies arise not just from pairs, but from interactions between groups of particles or individuals. This leads to the study of higher-order Ising models, specifically, the $p$-tensor Ising model, where the sufficient statistic is a multilinear polynomial of degree $p\geq 2$, for capturing higher-order interactions between the different particles. These models can be represented as a spin system on a $p$-uniform hypergraph, where the individual entities represent the vertices of the hypergraph and the $p$-tuples of interactions are indexed by the hyperedges.  More formally, given a vector of binary outcomes $\bm X := (X_1,\ldots,X_N) \in \cC_N:=\{-1, 1\}^N$ and a $p$-tensor $\bm J_N = ((J_{i_1 \ldots i_p}))_{1 \leq i_1 \ldots i_p \leq N}$, the $p$-tensor Ising model is a discrete exponential family with probability mass function: 
\begin{equation}\label{model}
\p_{\beta,p}(\bm X) = \frac{1}{2^N Z_N(\beta,p)} e^{\beta H_N(\bm X)}  ,
\end{equation}
where the sufficient statistic (Hamiltonian)
\begin{equation}\label{eq:HN}
H_N(\bm X) := \sum_{1\leq i_1,\ldots,i_p \leq N} J_{i_1\ldots i_p} X_{i_1}\ldots X_{i_p}, 
\end{equation} 
and $\beta \geq 0$ is the {\it natural parameter} (referred to as the {\it inverse temperature} in statistical physics) of the model. The  {\it normalizing constant} $Z_N(\beta,p)$ (also referred to as the {\it partition function}) is determined by the condition $\sum_{\bm X \in \mathcal{C}_N} \p_{\beta,p}(\bm X) = 1$, that is,  
\begin{equation*}
Z_N(\beta,p) = \frac{1}{2^N} \sum_{\bm X \in \mathcal{C}_N} \exp\left\{ \beta \sum_{1\leq i_1,\ldots,i_p \leq N} J_{i_1\ldots i_p} X_{i_1}\ldots X_{i_p} \right\}
\end{equation*} 
We will denote by $F_N(\beta,p) := \log Z_N(\beta,p)$ the {\it log-partition} function of the model. Higher-order Ising models  arise naturally in the study of multi-atom interactions in lattice gas models, such as the square-lattice eight-vertex model, the Ashkin-Teller model, and Suzuki's pseudo-3D anisotropic model  (cf.~  \cite{ab_ferromagnetic_pspin,multispin_simulations,ferromagnetic_mean_field,ising_partition_approximation,pspinref1,ising_suzuki,ising_general,turban,pspinref2} and the references therein).  More recently, higher-order spin systems  have also been proposed for modeling peer-group effects in social networks \cite{cd_ising_II}.

Hereafter, unless mentioned otherwise, we will assume that the tensor $\bm J_N$ satisfies the following two properties: 
\begin{itemize} 

\item[(1)] The tensor $\bm J_N$ is {\it symmetric}, that is, $J_{i_1\ldots i_p} = J_{i_{\sigma(1)}\ldots i_{\sigma(p)}}$ for every $1\leq i_1<\cdots<i_p\leq N$ and every permutation $\sigma$ of $\{1,\ldots,p\}$,  and 

\item[(2)] The tensor $\bm J_N$ has zeros on the `diagonals', that is, $J_{i_1\ldots i_p} = 0$, if $i_s = i_t$ for some $1\leq s<t\leq p$.

\end{itemize} 
In this paper, we consider the problem of estimating the parameter $\beta$ given a single sample $\bm X = (X_1, X_2, \ldots, X_n)$ from the $p$-tensor Ising model \eqref{model}. Extending the results of Chatterjee \cite{chatterjee} on MPL estimation in matrix ($p=2$) Ising models, we obtain a general theorem which gives conditions under which the MPL estimate is $\sqrt N$-consistent in the $p$-tensor Ising model, for any $p \geq 3$.\footnote{A sequence of estimators $\{\hat{\beta}_N\}_{N\geq 1}$ is said to be {\it consistent} at $\beta$, if $\hat{\beta}_N \pto \beta$ under $\P_{\beta}$, that is, for every $M > 0$, $\P_{\beta}(|\hat \beta_N(\bm X) - \beta| \leq M ) \rightarrow 1$ as $N \rightarrow \infty$. Moreover, a sequence of estimators $\{\hat{\beta}_N\}_{N\geq 1}$ is said to be $\sqrt N$-{\it consistent} at $\beta$, if for every $\delta > 0$, there exists $M:=M(\delta, \beta) > 0$ such that $\P_{\beta}(\sqrt N |\hat \beta_N(\bm X) - \beta| \leq M ) > 1-\delta$, for all $N$.} The main bottleneck in extending the results from the matrix to the tensor case, is the lack of a natural spectral condition that is strong enough to control the fluctuations of the MPL function, but still verifiable in natural examples. To this end, we introduce the notion of a {\it local interaction matrix} which, given a configuration $\bm x \in \{-1, 1\}^n$,  measures the strength of the interaction between pairs of vertices (Definition \ref{defn:JN_interaction}). Our result shows that the MPL estimate is $\sqrt N$-consistent, whenever we have an appropriate moment bound on the local interaction matrix, and if the normalized log-partition function stays bounded away from zero (Theorem \ref{chextension}). We illustrate the robustness and generality of our result by verifying the conditions of the theorem in various commonly studied tensor Ising models. This includes the $\sqrt N$-consistency of the MPL estimate in the well-known $p$-spin Sherrington-Kirkpatrick (SK) model \cite{bovier,panchenko_book} (Corollary \ref{skthreshold}), and in Ising models on $p$-uniform hypergraphs under appropriate conditions on the adjacency tensors (Corollary \ref{boundeddeg}). The latter is also related to the recent work of Daskalakis et al. \cite{cd_ising_II}, where a general model  for logistic regression with dependent observations using higher-order Ising models was proposed, which includes as a special case the model in \eqref{model}. However, the conditions in \cite{cd_ising_II} are based directly on the interaction tensor, hence, cannot handle models where the rate of  estimation undergoes a phase transition. This is understandable because  \cite{cd_ising_II} considered the problem of jointly estimating multiple parameters in a more general model, hence, stronger assumptions were necessary for ensuring consistency. Our goal, on the other hand, is to pin down the precise conditions necessary for estimating the single parameter $\beta$ and develop methods for verifying those conditions in natural examples. To this end, our general theorem recovers as a corollary, the results in \cite{cd_ising_II} when specialized to the model \eqref{model}.  More importantly, our results can handle models where the rate of estimation has phase transitions, which happens whenever the underlying hypergraph becomes dense.  To illustrate this phenomenon we consider the Ising model on a hypergraph stochastic block model (HSBM), a natural generalization of the widely studied (graph) stochastic block model, that serves as a natural model for capturing higher-order relational data \cite{hypergraph_learning,hypergraph_image,hypergraph_multimedia,hypergraph_gene}.  In this case, we show there is a critical value $\beta_{\mathrm{HSBM}}^*$, such that if $\beta > \beta_{\mathrm{HSBM}}^*$ then the MPL estimate is $\sqrt N$-consistent, while if $\beta < \beta_{\mathrm{HSBM}}^*$ there is no consistent estimator for $\beta$ (Theorem \ref{sbmthr}). While it is relatively straightforward to show the $\sqrt N$-consistency of the MPL estimate above the threshold using our general theorem, proving that estimation is impossible below the threshold is more challenging. This is one of the technical highlights of the paper, which requires careful combinatorial estimates that go beyond the standard mean-field approximation techniques. 

Next, we consider the special case of the $p$-tensor Curie-Weiss model, which is the Ising model on the complete $p$-uniform hypergraph. Here, using the special structure of the interaction tensor we are able to obtain the exact limiting distribution of the MPL estimate for all points above the estimation threshold (Theorem \ref{thm:cwmplclt}). In fact, in this regime the asymptotic variance of the MPL estimate saturates the Cramer-Rao lower bound, that is, the MPL estimate attains the best asymptotic variance among the class of consistent estimates. Finally, we derive the asymptotic distribution of the MPL estimate in the $p$-tensor Curie-Weiss model at the estimation threshold (this is the point below which consistent estimation is impossible and above which the MPL estimate is $\sqrt N$-consistent). Interestingly, at the threshold the asymptotics of the MPL estimate depend on the value of $p$. In particular, the MPL estimate is $\sqrt N$-consistent (with a non-Gaussian limiting distribution) when $p=2$, but inconsistent for $p \geq 3$.  The formal statements of the results and their various consequences are given below in Section \ref{sec:statements}.

\subsection{Related Work on Structure Learning}

The problem of structure learning in Markov random fields is another related area of active research, where the goal is to estimate the underlying graph structure given access to {\it multiple} i.i.d. samples from an Ising model, or a more general graphical model. For more on these results refer to  \cite{structure_learning,bresler,discrete_tree,graphical_models_algorithmic,highdim_ising,graphical_models_binary} and the references therein. Recently, Daskalakis et al. \cite{cd_testing} studied the related problems of identity and independence testing, and Neykov et al. \cite{high_tempferro,neykovliu_property} considered problems in graph property testing,  given access to multiple samples from an Ising model.  All these results, however, are in contrast with the present work, where the underlying graph structure is assumed to be known and the goal is to  estimate the natural parameters given a {\it single} sample from the model. This is motivated by the applications mentioned above where it is often difficult, if not impossible, to generate many independent samples from the model within a reasonable amount of time.

\section{Statements of the Main Results}\label{sec:statements} 

In this section we state our main results. The general result about the $\sqrt N$-consistency of the MPL estimate in tensor Ising models is discussed in Section \ref{sec:2.1}. Applications of this result to the  $p$-spin   SK model and Ising models on various hypergraphs are discussed in Section \ref{sec:2.2}.  Finally, in Section \ref{cltsec} we obtain the limiting distribution of the MPL estimate in the $p$-spin   Curie-Weiss model. Hereafter, we will often omit the dependence on $p$ and abbreviate $\p_{\beta,p}$ $Z_N(\beta,p)$, and $F_N(\beta,p)$ by $\p_\beta$, $Z_N(\beta)$, and $F_N(\beta)$, respectively, when there is no scope of confusion.

\subsection{Rate of Consistency of the MPL Estimator}\label{sec:2.1}

The maximum pseudo-likelihood (MPL) method, introduced by Besag \cite{besag_lattice,besag_nl}, provides a way to conveniently approximate the joint distribution of $\bm X \sim \P_{\beta, p}$ that avoids calculations with the normalizing constant. 

\begin{defn}\cite{besag_lattice,besag_nl} {\em Given a discrete random vector $\bm X= (X_1, X_2, \ldots, X_N)$ whose joint distribution is parameterized by a parameter $\beta \in \R$,  the MPL estimate of $\beta$ is defined as
\begin{equation*}
\hat {\beta}_N(\bm X):=\arg\max_{\beta \in \R}\prod_{i=1}^N f_i(\beta, \bm X),
\end{equation*}
where $f_i(\beta, \bm X)$ is the conditional probability mass function of $X_i$ given $(X_j)_{j \ne i}$. }
\end{defn}

To compute the MPL estimate in the $p$-tensor Ising model \eqref{model}, fix $\beta > 0$ and consider $\bm X \sim \P_{\beta}$. Then from  \eqref{model}, the conditional distribution of $X_i$ given $(X_j)_{j\neq i}$ can be easily computed as: 
\begin{equation}\label{eq:conditional}
\P_{\beta}\left(X_i\big|(X_j)_{j\neq i}\right) = \frac{e^{p \beta  X_i m_i(\bm X) }}{ e^{p \beta  m_i(\bm X)  } + e^{-p \beta  m_i(\bm X)  } },
\end{equation}
where $m_i(\bm X) := \sum_{1\leq i_2,\ldots,i_p \leq N} J_{i i_2 \ldots i_p} X_{i_2}\cdots X_{i_p}$, is the local effect at the node $1 \leq i \leq N$ (often referred to as the local magnetization of the vertex $i$ in the statistical physics literature).  Then the pseudolikelihood estimate of $\beta$ (as defined in \eqref{model}) in the $p$-tensor Ising model \eqref{eq:conditional} is obtained by maximizing the function below, with respect to $b$,
\begin{align*}
L(b|\bm X) := \prod_{i=1}^N \P_{b}\left(X_i\big|(X_j)_{j\neq i}\right) & = \frac{1}{2^N} \exp\left\{\sum_{i=1}^N \left\{ p b   X_i m_i(\bm X) -  \log \cosh\left(p b  m_i(\bm X) \right) \right\} \right\} .
 \end{align*}
Now, since $\log L(b|\bm X)$ is concave in $b$, the MPL estimator $\hat{\beta}_N(\bm X)$ can be obtained by solving the gradient equation $\frac{\partial \log L(b|\bm X)}{\partial b} = 0$, which simplifies to 
\begin{equation}\label{solution}
H_N(\bm X) -  \sum_{i=1}^N m_i(\bm X) \tanh\left(p b m_i(\bs)\right) = 0.
\end{equation} 
To ensure well-definedness, in case \eqref{solution} does not have a solution or has more than one solution, the MPL estimate $\hat{\beta}_N(\bm X)$ is more formally defined as: 
 \begin{equation}\label{mple}
 \hat{\beta}_N(\bm X):= \inf\left\{b \geq 0 : H_N(\bm X) =  \sum_{i=1}^N m_i(\bm X) \tanh\left( p b m_i(\bs)\right)\right\},
 \end{equation}
where the infimum of an empty-set is defined to be $+\infty$. Note that the expression in the RHS of the equality in \eqref{mple} is an increasing function of $t$, hence $\hat{\beta}_N(\bm X)$ can be very easily computed by the Newton-Raphson method or even a simple grid search.

Our first result is about the rate of consistency of the MPL estimate in general tensor Ising models. In particular,  we show in the proposition below that the MPL estimate $\hat \beta_N(\bm X)$, based on a single sample $\bm X\sim \P_{\beta}$ converges to the true parameter $\beta$ at rate $1/\sqrt N$, whenever the interaction tensor $\bm J_N$ satisfies a certain spectral-type condition and the log-partition function is $\Omega(N)$\footnote{For positive sequences $\{a_n\}_{n\geq 1}$ and $\{b_n\}_{n\geq 1}$, $a_n = O(b_n)$ means $a_n \leq C_1 b_n$, $a_n =\Omega(b_n)$ means $ a_n \geq C_2 b_n$, and  $a_n = \Theta(b_n)$ means $C_1 b_n \leq a_n \leq C_2 b_n$, for all $n$ large enough and positive constants $C_1, C_2$. Moreover, subscripts in the above notation,  for example $O_\square$, denote that the hidden constants may depend on the subscripted parameters.} at the true parameter value. To state our result formally, we need the following definition:

\begin{defn}\label{defn:JN_interaction} {\em Given a $p$-tensor $\bm J_N=((J_{i_1 i_2\ldots i_p}))_{1 \leq i_1, i_2, \ldots, i_p \leq N}$ and $\bm x = (x_1, x_2, \ldots, x_N) \in \cC_N$, define the {\it local interaction matrix} of $\bm J_N$ at the point $\bm x$ as the $N \times N$ matrix 
$\bm J_N(\bt) := ((J_{i_1 i_2}(\bt)))_{1\leq i_1, i_2 \leq N}$, where the entries are given by: 
\begin{align}\label{eq:Jx}
J_{i_1 i_2}(\bt) := \sum_{1\leq i_3,\ldots,i_p \leq N}J_{i_1 i_2 i_3\ldots i_p}x_{i_3}\cdots x_{i_p}. 
\end{align} 
(Note that in the case $p=2$, $J_{i_1 i_2}(\bt) =J_{i_1 i_2}$, that is, the  local interaction matrix $\bm J_N(\bm x)$  is same as the interaction matrix $\bm J_N$, for all $\bm x \in \cC_N$.) }
\end{defn}

We are now ready to state our result on the convergence rate of the MPL estimate in a tensor Ising model.\footnote{For a vector $\bm v \in \R^N$, $\|\bm v\|$ will denote the Euclidean norm of $\bm v$. Moreover, for a $N \times N$ matrix $A$, $\|A\| := \sup_{\|\bm x\|=1} \|A\bm x\|$ denotes the operator norm of $A$.} 

\begin{thm}\label{chextension}
Fix $p \geq 2$, $\beta > 0$ and a sequence of $p$-tensors $\{\bm J_N\}_{N \geq 1}$ such that the following two conditions hold: 
\begin{enumerate}
\item[$(1)$] $\sup_{N\geq 1}\E_{\beta}[ \|\bm J_N(\bm Z)\|^4] <\infty$, where the expectation is taken with respect to $\bm Z \sim \P_{\beta}$, 
\item[$(2)$] $\liminf_{N\rightarrow \infty} \frac{1}{N}F_N(\beta) > 0$.
\end{enumerate}
Then given a single sample $\bm X$ from the model \eqref{model} with interaction tensor $\bm J_N$, the MPL estimate $\hat{\beta}_N(\bm X)$, as defined in \eqref{mple}, is $\sqrt N$-consistent for $\beta$, that is, for every $\delta > 0$, there exists $M:=M(\delta, \beta) > 0$ such that $$\P_{\beta}(\sqrt N |\hat \beta_N(\bm X) - \beta| \leq M ) > 1-\delta,$$ for all $N$ large enough.  
\end{thm}

The proof of this theorem is given in Section \ref{sec:3}. The proof has two main steps: In the first step we use the method of exchangeable pairs to show that the derivative of the log-pseudolikelihood (the LHS of \eqref{solution}) is concentrated around zero at the true model parameter (see Lemma \ref{lm:boundsecondmoment} for details). The proof adapts the method of exchangeable pairs introduced in \cite{chatterjee} where a similar result was proved for matrix (2-spin) Ising models. The main technical challenge as one goes from the matrix to the tensor case, is the absence of a natural spectral condition in tensor models. To this end, we introduce condition (1), which requires that the fourth-moment of the spectral norm of the local interaction matrix is uniformly bounded. This condition allows us to prove the desired concentration of the  log-pseudolikelihood, and, as we will see below, can be easily verified for a large class of natural tensor models. The second step in the proof of Theorem \ref{chextension} is to show that the log-pseudolikelihood is strongly concave, that is, its second derivative is strictly negative with high probability. Here, we use condition (2) to first show that the Hamiltonian is $\Omega(N)$ with high-probability, which then implies the strong concavity of the log-pseudolikelihood by a truncated second-moment argument.\footnote{Recalling the discussion in Definition \ref{defn:JN_interaction}, note that when $p=2$, condition $(1)$ simplifies to $\sup_{N \geq 1}||\bm J_N|| < \infty$, hence Theorem \ref{chextension} recovers Chatterjee's result on $\sqrt N$-consistency of MPL estimates in 2-spin Ising models \cite[Theorem 1.1]{chatterjee}.}

\begin{remark}\label{remark:condition} {\em The $L_4$-condition (condition (1)) on the local interaction matrix in Theorem \ref{chextension} can be replaced by the following stronger $L_\infty$-condition, which is often easier to verify in examples: 
\begin{equation}\label{eq:JNcondition}
\sup_{N\geq 1}\sup_{\bt \in \sa_N} \|\bm J_N(\bt)\| <\infty.
\end{equation} 
Condition \eqref{eq:JNcondition}, hence condition $(1)$ in Theorem \ref{chextension}, is also weaker than the `bounded-degree condition': 
\begin{align}\label{eq:JNcondition_II}
\sup_{1 \leq i_1 \leq N} \sum_{1 \leq i_2, i_3, \ldots, i_p \leq N} |J_{i_1 i_2 i_3 \ldots i_p} | =O(1). 
\end{align} 
In particular, condition \eqref{eq:JNcondition} allows us to handle the $p$-spin Sherrington-Kirkpatrick model, an example where the bounded-degree condition \eqref{eq:JNcondition_II} fails to hold. } 
\end{remark}

\subsection{Applications}\label{sec:2.2}
 
In this section we discuss the consequences of Theorem \ref{chextension} to the $p$-spin SK model (Section \ref{sec:sk}), spin systems of on general hypergraphs (Section \ref{nonstoch}), and the hypergraph stochastic block model (Section \ref{stoch}).

\subsubsection{The $p$-Spin Sherrington-Kirkpatrick Model}
\label{sec:sk}

In the $p$-spin Sherrington-Kirkpatrick (SK) model \cite{bovier}, the interaction tensor is of the form  
\begin{align}\label{eq:JN_sk}
J_{i_1\ldots i_p} = N^{\frac{1-p}{2}} g_{i_1\ldots i_p}, 
\end{align}
where  $(g_{i_1\ldots i_p})_{1\leq i_1<\ldots<i_p< \infty}$ is a fixed realization of a collection of independent standard Gaussian random variables, and $g_{i_1\ldots i_p} = g_{\sigma(i_1)\ldots \sigma(i_p)}$, for any permutation $\sigma$ of $\{1, 2, \ldots, p\}$. This is a canonical example of a spin glass model which has remarkable thermodynamic properties \cite{spinglass_book}. A whole new discipline has emerged from the study of this object, with many beautiful theorems that have unearthed deep connections between diverse areas in mathematics and statistical physics  (cf.~\cite{bovier,sk_optimization,panchenko_book,talagrand_sk} and the references therein). The problem of parameter estimation in the SK model was initiated by Chatterjee \cite{chatterjee}, where $\sqrt N$-consistent of the MPL estimate for all $\beta > 0$ was proved for the 2-spin SK model. The following corollary extends this to all $p \geq 3$.

\begin{cor}\label{skthreshold} In the $p$-spin SK model, the MPL estimate $\hat{\beta}_N(\bm X)$  is $\sqrt{N}$-consistent for all $\beta> 0$. 
\end{cor}

The proof of this result is given in Section \ref{skproof}. In this case, condition (2) in Theorem \ref{chextension} can be easily verified using monotonicity and the well-known asymptotics of $F_N(\beta)$ in the high-temperature (small $\beta$) regime: In particular, we know from \cite[Theorem 1.1]{bovier} that, almost surely, $\lim_{N \rightarrow \infty}\frac{1}{N}F_N(\beta) = \tfrac{\beta^2}{2}$, for $\beta >0$ small enough. Hence, by the monotonicity of $F_N(\beta)$, we have $\lim_{N \rightarrow \infty}\frac{1}{N}F_N(\beta) > 0 $ for all $\beta > 0$, which establishes (2).  However, unlike when $p=2$, verifying condition $(1)$ in Theorem \ref{chextension} when $p \geq 3$ requires more work.\footnote{Note that when $p=2$, $\bm J_N$ is a Wigner matrix, and hence, by \cite[Theorem 2.12]{spectral_norm}  $\sup_{N \geq 1} ||\bm J_N|| < \infty$, thus verifying condition $(1)$ of Theorem \ref{chextension}.} To this end, note that for $p \geq 3$ and every fixed $\bm x \in \cC_N$, the local interaction matrix $\bm J_N(\bm x)$ is a Gaussian random matrix, but the elements are now dependent because of the symmetry of the tensor $\bm J_N$. This dependence, however, is relatively weak and using standard Gaussian process machinery we can show the validity of \eqref{eq:JNcondition}, and, hence, that of condition $(1)$ in Theorem \ref{chextension}.

\subsubsection{Ising Models on Hypergraphs}\label{nonstoch}

The $p$-tensor model \eqref{model} can be interpreted as a spin system on a weighted $p$-uniform hypergraph, where the entries of the tensor correspond to the weights of the hyperedges. More precisely, given a symmetric tensor $\bm J_N = ((J_{i_1 i_2 \ldots i_p}))_{1 \leq i_1, i_2, \ldots, i_p \leq N}$, construct a weighted $p$-uniform hypergraph $H_N$ with vertex set $[N]:=\{1, 2, \ldots, N\}$ and edge weights $w(\bm e)= J_{i_1 i_2 \ldots i_p}$, for $\bm e =(i_1, i_2, \ldots, i_p) \in [N]_p$.\footnote{For the set $[N]=\{1, 2, \ldots, N\}$, $[N]^p$ denotes the $p$-fold Cartesian product $[N]\times [N] \times \cdots \times [N]$, and $[N]_p$ is the collection of $p$-tuples in $[N]^p$ with distinct entries.} The model \eqref{model} is then a spin system on $H_N$ where the Hamiltonian \eqref{eq:HN} can be rewritten as 
$$H_N(\bm X) = \sum_{\bm e \in {[N]_p}} w(\bm e) X_{\bm e},$$
where $\bm X= (X_1, X_2, \ldots, X_N) \in \cC_N$ and $X_{\bm e} =X_{i_1} X_{i_2} \ldots X_{i_p}$, for $\bm e = (i_1, i_2, \ldots, i_p)$. For a tensor $\bm J_N=((J_{i_1 i_2 \ldots i_N}))$, define the (weighted) degree of the vertex $i_1$ as $$d_{\bm J_N}({i_1}):= \frac{1}{(p-1)!}\sum_{1\leq i_2, i_3,\ldots,i_p \leq N}|J_{i_1 i_2 i_3\ldots i_p}|,$$ 
which is the sum of the absolute values weights of the hyperedges passing through the vertex $i_1$. Similarly, define the weighted {\it co-degree} of the vertices $i_1, i_2$ as 
\begin{align}\label{eq:dJN}
d_{\bm J_N}(i_1, i_2):=\frac{1}{(p-2)!}\sum_{1\leq i_3,\ldots,i_p \leq N} |J_{i_1 i_2 i_3\ldots i_p}|, 
\end{align}
which is the sum of the absolute values of weights of the hyperedges incident on both $i_1$ and $i_2$. Denote by $\bm D_{\bm J_N} = ((d_{\bm J_N}(i_1, i_2)))_{1 \leq i_1, i_2 \leq N}$, the {\it co-degree matrix} corresponding to the tensor $\bm J_N$. The following corollary provides useful sufficient conditions under which the MPL estimate is $\sqrt N$-consistent at all   temperatures. The proof is given in Section \ref{sec:boundeddegpf}.

\begin{cor}\label{boundeddeg} Suppose $\{\bm J_N\}_{N \geq 1}$ is a sequence of $p$-tensors such that the following two conditions hold: 
\begin{enumerate}
\item[$(1)$] $\sup_{N\geq 1} \|\bm D_{\bm J_N} \| <\infty$,
\item[$(2)$] $\liminf_{N\rightarrow \infty} \frac{1}{N} \sum_{1 \leq i_1 < i_2 < \ldots < i_p \leq N} J_{i_1 i_2 \ldots i_p}^2> 0$.
\end{enumerate}
Then the MPL estimate $\hat{\beta}_N(\bm X)$ is $\sqrt N$-consistent for all $\beta > 0$. 
\end{cor}

\begin{remark}\label{remark:condition_II} 
{\em Note that, since the $L_2$-operator norm of a symmetric matrix is bounded by its $L_\infty$-operator norm,\footnote{For two sequences $a_n$ and $b_n$, $a_n \lesssim_{\Box} b_n$ means that there exists a positive constant $C(\Box)$ depending only on the subscripted parameters $\Box$, such that $a_n \leq C(\Box) b_n$ for all $n$ large enough.}     
\begin{align}\label{eq:DN_bound}
\|\bm D_{\bm J_N} \| \leq \max_{1 \leq i_1 \leq N} \sum_{i_2=1}^N d_{\bm J_N}(i_1, i_2) & = \frac{1}{(p-2)!} \max_{1 \leq i_1 \leq N} \sum_{1 \leq i_2, i_3, \ldots, i_p \leq N} |J_{i_1 i_2 \ldots i_p}| \nonumber \\   
& \lesssim_p \max_{1 \leq i_1 \leq N} d_{\bm J_N}(i_1), 
\end{align}
that is, if a tensor has bounded maximum degree, then condition $(1)$ of Theorem \ref{chextension} holds. This shows that Corollary \ref{boundeddeg} recovers the general theorem of \cite{cd_ising_II}, where $\sqrt N$-consistency of the MPL was proved, albeit for a more general model, under condition (2) and condition $(1)$ replaced by the bounded degree assumption $\max_{1 \leq i_1 \leq N} d_{i_1} = O(1)$. }   
\end{remark}

As mentioned earlier, the conditions in Corollary \ref{boundeddeg}, neither of which depend on the true parameter $\beta$, cannot hold for hypergraphs where the rate of estimation undergoes a phase transition. In fact, as explained in Remark \ref{remark:hypergraph}, the scope of this corollary is really only restricted to Ising models on hypergraphs which are sparse. The importance of the second condition in  Theorem \ref{chextension} becomes evident when the  hypergraph becomes dense, where $F_N(\beta)$ ceases to be $\Omega(N)$ for all $\beta$, and the rate of estimation changes as $\beta$ varies. This is illustrated in Section \ref{stoch} below, where the exact location of the phase transition is derived for Ising models on block hypergraphs.

\begin{remark}\label{remark:hypergraph} {\em Suppose $H_N=(V(H_N), E(H_N))$ is a sequence of unweighted $p$-uniform hypergraphs with vertex set $V(H_N)=[N] = \{1, 2, \ldots, N\}$ and edge set $E(H_N)$, with no isolated vertex. Denote by $\bm A_{H_N} = ((a_{i_1 i_2 \ldots i_p}))_{1 \leq i_1, i_2, \ldots, i_p \leq N}$ the adjacency tensor of $H_N$, that is, $a_{i_1 i_2 \ldots i_p} =1$ if $(i_1, i_2, \ldots, i_p) \in E(H_N)$ and zero otherwise. Then in order to ensure that a $p$-spin-system on $H_N$, as in \eqref{model}, has a non-trivial scaling limit, one needs to consider the scaled tensor,  
$$\bm J_{H_N}= \frac{N}{|E(H_N)|} \bm A_{H_N}.$$
In this case, the Frobenius norm condition in Corollary \ref{boundeddeg} simplifies to,   
\begin{align}\label{eq:JHN_condition}
\frac{1}{N} ||\bm J_{H_N}||_F^2 = \frac{N}{|E(H_N)|^2} \sum_{1 \leq i_1, i_2, \ldots, i_p \leq N} a_{i_1 i_2 
\ldots i_p} = \Theta\left(\frac{N}{|E(H_N)|} \right) = \Omega(1).
\end{align}
This implies, $|E(H_N)| = \Theta(N)$, since $H_N$ has no isolated vertex. Moreover, condition $(1)$ can be written as,  
\begin{align}\label{eq:DHN}
||\bm D_{\bm A_{H_N} }|| = O\left(\frac{|E(H_N)|}{N} \right). 
\end{align}
Therefore, combining \eqref{eq:JHN_condition}, \eqref{eq:DHN}, and Corollary \ref{boundeddeg}, shows that for any sequence of (unweighted) $p$-uniform hypergraphs $H_N= (V(H_N), E(H_N))$, such that $||\bm D_{\bm A_{H_N} }|| = O(1)$ and $|E(H_N)| = O(N)$, the MPL estimate $\hat{\beta}_N(\bm X)$ in the Ising model \eqref{model} with interaction tensor $\bm J_{H_N}$,  is $\sqrt N$-consistent for all $\beta > 0$.  In particular, by the bound in \eqref{eq:DN_bound} applied to the adjacency tensor $\bm A_{H_N}$, the MPL estimate $\hat{\beta}_N(\bm X)$ is $\sqrt N$-consistent for all $\beta > 0$, whenever $H_N$ has bounded maximum degree and $O(N)$ edges. } 
\end{remark}

\subsubsection{Hypergraph Stochastic Block Models}\label{stoch}

The hypergraph stochastic block model (HSBM) is a random hypergraph model where each hyperedge is present independently with probability depending on the membership of the vertices to various blocks (see \cite{hypergraph_block,hypergraph_clustering_II,hypergraph_phase_transition} and the references therein for more on the HSBM and its applications in higher-order community detection).

\begin{defn}\label{defn:block} {\em (Hypergraph Stochastic Block Model) Fix $p \geq 2$, $K\geq 1$,  a vector of community proportions $\bm \lambda := (\lambda_1,\ldots,\lambda_K)$, such that $\sum_{j=1}^K \lambda_j = 1$, and a symmetric probability tensor $\bm \Theta := ((\theta_{j_1\ldots j_p}))_{1\leq j_1,\ldots,j_p\leq K}$,  where $\theta_{j_1\ldots j_p} \in [0, 1]$, for $1\leq i_1,\ldots, i_p\leq K$. The hypergraph stochastic block model with proportion vector $\bm \lambda$ and probability tensor $\bm \Theta$ is a $p$-uniform hypergraph $H_N$ on $[N]=\{1, 2, \ldots, N\}$ vertices with adjacency tensor $\bm A_{H_N}=((a_{i_1 i_2 \ldots i_p}))_{1 \leq i_1, i_2, \ldots, i_p \leq N}$, where 
$$a_{i_1\ldots i_p} \sim \mathrm{Ber}\left(\theta_{j_1\ldots j_p}\right)\quad\textrm{for } i_1<\ldots<i_p \textrm{ and }  (i_1,\ldots,i_p) \in \cB_{j_1}\times\cdots\times \cB_{j_p},$$ where $\cB_j := (N \sum_{i=1}^{j-1}\lambda_i, N \sum_{i=1}^j \lambda_i] \bigcap [N]$, for $j \in \{1,\ldots,K\}$, and $\{a_{i_1\ldots i_p}\}_{1\leq i_1<\ldots<i_p \leq |V|}$ are independent. We denote this model by $\cH_{p, K, N} (\bm \lambda, \bm \Theta)$ and a realization from this model as $H_N \sim \cH_{p, K, N} (\bm \lambda, \bm \Theta)$. }
\end{defn}

In this section, we consider the problem of parameter estimation given a sample from an Ising model on a HSBM.  The following theorem shows that for the $p$-tensor Ising models on a HSBM, there is a critical value of $\beta$, below which estimation is impossible, and above which the MPL estimate is $\sqrt N$-consistent. The location of the phase transition is determined by the first time the maximum of a certain variational problem, which arises from the mean-field approximation of the partition function,  becomes non-zero. More formally, this is defined as,   
\begin{equation}\label{eq:beta_threshold}
\beta_{\mathrm{HSBM}}^* :=  \sup\left\{\beta \geq 0: \sup_{(t_1,\ldots,t_K)\in [0,1]^K} \phi_\beta(t_1,\ldots,t_K) = 0\right\},
\end{equation}
where the function 
$\phi_\beta: [-1,1]^K \mapsto \mathbb{R}$ is: 
\begin{align}\label{eq:threshold_function}
\phi_{\beta}(t_1, t_2, \ldots, t_K) := \beta \left(\sum_{1\leq j_1,\ldots,j_p\leq K} \theta_{j_1\ldots j_p} \prod_{\ell=1}^p \lambda_{j_\ell} t_{j_\ell} \right) - \sum_{j=1}^K \lambda_j I(t_j),
\end{align}
and $I(t) := \frac{1}{2}\left\{(1+t)\log(1+t) + (1-t)\log(1-t)\right\}$ is the binary entropy function.

\begin{thm}\label{sbmthr}  Fix $p \geq 2$ and a realization of a HSBM $H_N \sim \cH_{p, K, N}(\bm \lambda, \bm \Theta)$ on $N$ vertices, where $\bm \lambda$ is a proportion vector and $\bm \Theta$ is a symmetric probability tensor as in Definition \ref{defn:block}.  Then given a sample $\bm X \sim \P_{\beta}$ from the model \eqref{model}, with adjacency tensor $\bm J_N= \frac{1}{N^{p-1}} \bm A_{H_N}$, the following hold: 
\begin{itemize}

\item[(1)] The MPL estimate $\hat{\beta}_N(\bm X)$ is $\sqrt{N}$-consistent for $\beta > \beta_{\mathrm{HSBM}}^*$. 

\item[(2)] There does not exist any consistent sequence of estimators for any $\beta < \beta_{\mathrm{HSBM}}^*$.

\end{itemize}
\end{thm}

The proof of the above result is given in Section \ref{proof6}. To show the result in $(1)$ we verify the conditions of Theorem \ref{chextension}. Here, we invoke the standard mean-field lower bound to the Gibbs variational representation of the partition function \cite{CD16}, from which it can be easily verified that $F_N(\beta) = \Omega(N)$, whenever $\beta > \beta_{\mathrm{HSBM}}^*$. Perhaps the more interesting consequence of Theorem \ref{sbmthr} is the result in (2), which shows that not only is the MPL estimate not $\sqrt N$-consistent below the threshold, no estimator is consistent in this regime, let alone $\sqrt N$-consistent. The main argument in this proof is to show that 
\begin{align}\label{eq:Fbeta_bounded_I}
F_N(\beta) = O(1), \quad \text{ for } \beta < \beta_{\mathrm{HSBM}}^*. 
\end{align} 
Once this is proved, then it can be easily verified that the Kullback-Leibler (KL) divergence between the measures $\P_{\beta_1, p}$ and $\P_{\beta_2, p}$, for any two $0 < \beta_1 < \beta_2 < \beta_{\mathrm{HSBM}}^*$ remains bounded, which in turn implies that the measures $\P_{\beta_1, p}$ and $\P_{\beta_2, p}$ are untestable, and hence inestimable. The main technical difficulty in proving an estimate like \eqref{eq:Fbeta_bounded_I} in tensor models, is the absence of `Gaussian' techniques \cite{BM16,comets}, which allows one to compare the partition function of Ising models with quadratic Hamiltonians with an appropriately chosen Gaussian model. This method, unfortunately, does not apply when $p \geq 3$, hence, to estimate the partition function we take the following more direct approach: We first consider the {\it averaged model} where the interaction tensor is replaced by the expected interaction tensor $\tilde{\bm J}_N := \e {\bm J_N}$. Using the block structure of the tensor $\tilde{\bm J}_N$ the Hamiltonian in the averaged model can be written in terms of the average of the spins in the different blocks, and hence, the partition function in the averaged model can be accurately estimated using bare-hands combinatorics (Lemma \ref{lm:logpartition_I}). We then move from the averaged model to the actual model using standard concentration arguments (Lemma \ref{st2}).

\begin{remark} {\em Using the machinery of non-linear large deviations developed in \cite{CD16}, we can in fact show that for the HSBM,  
\begin{align}\label{Fbeta_variational_problem}
\lim_{N \rightarrow \infty}\frac{1}{N} F_N(\beta) = \sup_{(t_1,\ldots,t_K)\in [0,1]^K} \phi_\beta(t_1,\ldots,t_K) , 
\end{align}
with probability 1.  Although the proof of this result has not been included in the paper, because for proving Theorem \ref{sbmthr} $(1)$ we only need to establish a lower bound on $\frac{1}{N} F_N(\beta)$, this is worth mentioning as it motivates the definition of the threshold $\beta_{\mathrm{HSBM}}^*$ and corroborates the result in Theorem \ref{sbmthr} (1).  The result in \eqref{Fbeta_variational_problem} is, however, not strong enough to show that estimation is impossible below the threshold $\beta_{\mathrm{HSBM}}^*$. Here, we need to understand the asymptotic behavior of $F_N(\beta)$ itself (without scaling by $N$), which is a more delicate matter that require arguments beyond the purview of non-linear large deviations and mean-field approximations, as discussed above. In this case, the proof of Theorem \ref{sbmthr} (2) shows that whenever the log-partition function is $o(N)$, which happens when $\beta < \beta_{\mathrm{HSBM}}^*$, it is actually $O(1)$, and hence,  there is a sharp transition from inestimability to $\sqrt N$-consistency. } 
\end{remark}

An important special case of the HSBM is the Erd\H{o}s-R\'enyi random hypergraph model, where every hyperedge is present independently with the same fixed probability. 

\begin{example}\label{example:random_hypergraph} {\em (Erd\H{o}s-R\'enyi random hypergraphs) The HSBM reduces to the classical Erd\H{o}s-R\'enyi random $p$-hypergraph model when the number of blocks $K=1$. In this case, each hyperedge is present independently with probability $\theta \in (0, 1]$, and the variational problem \eqref{eq:beta_threshold} for the threshold simplifies to 
\begin{equation}\label{hypergraph_random}
\beta_{\mathrm{ER}}^*(p, \theta) :=  \sup\left\{\beta \geq 0: \sup_{t \in [0,1]} \left\{ \beta \theta t^p - I(t) \right\} = 0\right\}.
 \end{equation} 
We will denote this hypergraph model by $\mathscr{G}_p(N, \theta)$. In this case, Theorem \ref{sbmthr} gives the following:
\begin{itemize}

\item In the  Erd\H{o}s-R\'enyi random $p$-hypergraph model $\mathscr{G}_p(N, \theta)$, the MPL estimate $\hat{\beta}_N(\bm X)$ is $\sqrt{N}$-consistent for all $\beta > \beta_{\mathrm{ER}}^*(p, \theta)$. 

\item On the other hand, there does not exist any consistent sequence of estimators for any $\beta < \beta_{\mathrm{ER}}^*(p, \theta)$.  
\end{itemize} 
Note that by the change of variable $\kappa = \beta \theta$, it follows that $\beta_{\mathrm{ER}}^*(p, \theta) = \beta_{\mathrm{ER}}^*(p, 1)/\theta$. A simple analysis shows $\beta_{\mathrm{ER}}^*(2, 1)=0.5$, and hence, $\beta_{\mathrm{ER}}^*(2, \theta)  = 0.5/\theta$. For higher values of $p$, $\beta_{\mathrm{ER}}^*(p, 1)$ can be easily computed numerically. In particular, we have $\beta_{\mathrm{ER}}^*(3, 1) \approx 0.672$ and $\beta_{\mathrm{ER}}^*(4, 1) \approx 0.689$. In fact, $\beta_{\mathrm{ER}}^*(p, 1)$ is strictly increasing in $p$ and $\lim_{p \rightarrow \infty} \beta_{\mathrm{ER}}^*(p, 1) = \log 2$ (see Appendix \ref{cwtp} for a proof). } 
\end{example}

Another example is that of random $p$-partite $p$-uniform hypergraphs, which are natural extensions of random bipartite graphs.

\begin{example}[Random $p$-partite $p$-uniform hypergraphs] {\em A $p$-uniform hypergraph is said to be  $p$-partite if the vertex set of the hypergraph can be partitioned into $p$-nonempty sets in such a way that every edge intersects every set of the partition in exactly one vertex.  A random $p$-partite $p$-uniform hypergraph, is a  $p$-partite $p$-uniform hypergraph where each edge is present independently with some fixed probability $\theta \in (0, 1]$ \cite{multipartite_random_hypergraph}. More formally, given a vector $\bm N = (N_1, N_2, \ldots, N_p)$ of positive integers, such that $\sum_{j=1}^p N_j = N$ and $\theta \in (0, 1]$, in the random $p$-partite $p$-uniform hypergraph $\cH_p(\bm N, \theta)$, the vertex set  $[N]=\{1, 2, \ldots, N\}$ is partitioned into $p$ disjoint sets $S_1,\ldots, S_p$, such that $|S_j| = N_j$ for $1 \leq j \leq p$, and each edge $\bm e \in V_1 \times V_2 \times \cdots \times V_p$ is present independently with probability $\theta$.  If $\bm N$ is such that $\frac{1}{N} \bm N \rightarrow \bm \lambda = (\lambda_1, \lambda_2, \ldots, \lambda_p)$, as $N \rightarrow \infty$, then this is a special case of the hypergraph stochastic block model and the threshold \eqref{eq:beta_threshold} simplifies to,  
\begin{equation}\label{rurpthr}
	\beta_{\mathrm{partite}}^*(p, \bm \lambda, \theta) :=  \sup\left\{\beta \geq 0: \sup_{(t_1,\ldots,t_p)\in [0,1]^p}   \left\{ \beta \theta \prod_{j=1}^p \lambda_j t_j - \sum_{j=1}^p \lambda_j I(t_j) \right\}  = 0\right\}.
	\end{equation}
Theorem \ref{sbmthr} then implies that the MPL estimate is $\sqrt{N}$-consistent for all $\beta > \beta_{\mathrm{partite}}^*(p, \bm \lambda, \theta) $, and consistent estimation is impossible for $\beta < \beta_{\mathrm{partite}}^*(p, \bm \lambda, \theta) $. In case the $p$ partitioning sets have asymptotically equal size, that is, $\lambda_j = \frac{1}{p}$ for all $1 \leq j \leq p$, the threshold in \eqref{rurpthr} simplifies further to: 
\begin{align}\label{rurpthr_equal}
	\beta_{\mathrm{equipartite}}^*(p,  \theta) :=  \sup\left\{\beta \geq 0: \sup_{(t_1,\ldots,t_p)\in [0,1]^p}   \left\{ \beta \theta p^{-p} \prod_{j=1}^p  t_j - \frac{1}{p} \sum_{j=1}^p I(t_j) \right\}  = 0\right\}.
	\end{align}
Now, a simple analysis shows that $\beta_{\mathrm{equipartite}}^*(p,  \theta)=p^p\beta_{\mathrm{ER}}^*(p,  \theta)$. The upper bound $\beta_{\mathrm{equipartite}}^*(p,  \theta) \leq p^p\beta_{\mathrm{ER}}^*(p,  \theta)$ follows by substituting $t_1=t_2  \cdots = t_p = t \in [0, 1]$ in \eqref{rurpthr_equal} and relating it to \eqref{hypergraph_random}. For the lower bound, note by the convexity of the function $I(x)$ and the AM-GM inequality, that  
$$\beta \theta p^{-p} \prod_{j=1}^p  t_j - \frac{1}{p}\sum_{j=1}^p I(t_j) \leq \beta \theta p^{-p}\left(\frac{1}{p}\sum_{j=1}^p t_j \right)^p - I\left(\frac{1}{p}\sum_{j=1}^p t_j\right).$$
Then, by the change of variable $\kappa=\beta p^{-p}$, it follows that $\beta_{\mathrm{equipartite}}^*(p,  \theta) \geq p^p\beta_{\mathrm{ER}}^*(p,  \theta)$. } 
\end{example}

\subsection{Precise Fluctuations in the Curie-Weiss Model}\label{cltsec}
 
The $p$-tensor Curie-Weiss model is the Ising model on the complete $p$-uniform hypergraph,\footnote{In the complete $p$-uniform hypergraph with vertex set $[N]=\{1, 2, \ldots, N\}$ the set of hyperedges is the collection of all the $p$-element subsets of $[N]$.} where all the $p$-tuples of interactions are present \cite{ferromagnetic_mean_field}. In other words, this is the Ising model on the Erd\H os-R\'enyi $p$-hypergraph with $\theta=1$. Denoting $\beta_{\mathrm{CW}}^*(p) := \beta_{\mathrm{ER}}^*(p, 1)$, we know from the discussion in Example \ref{example:random_hypergraph}, that for  $\beta< \beta_{\mathrm{CW}}^*(p)$ consistent estimation is impossible, while for  $\beta > \beta_{\mathrm{CW}}^*(p)$ the MPL estimate $\hat{\beta}_N(\bm X)$  is $\sqrt{N}$-consistent. Given that we know the rate of consistency, the next natural question is to wonder whether anything can be said about the limiting distribution of the MPL estimate above the threshold. While tackling this question appears to be extremely difficult, if not impossible, for general models, the special structure of the Curie-Weiss model allows us to say much more. This begins with the observation that in the  Curie-Weiss model  the MPL estimate can be written as a function of the sample mean $\bar X_N = \frac{1}{N} \sum_{i=1}^N X_i$. Then combining the recent results on the   asymptotic distribution of $\bar X_N$ \cite{mlepaper} and the delta theorem, we can get the precise fluctuations of the MPL estimate at all points above the estimation threshold $\beta_{\mathrm{CW}}^*(p)$. This is formalized in the theorem below:

\begin{thm}\label{thm:cwmplclt} 
Fix $p \geq 2$ and consider the $p$-spin Curie-Weiss model with interaction tensor $\bm J_N = ((J_{i_1 \ldots i_p}))_{1 \leq i_1, \ldots, i_p \leq N}$, where $J_{i_1  \ldots i_p} = \frac{1}{N^{p-1}}$, for all $1 \leq i_1, \ldots, i_p \leq N$. Then for every $\beta > \beta_{\mathrm{CW}}^*(p)$, as $N \rightarrow \infty$,
\begin{equation}\label{statement1}
			\sqrt{N}(\hat{\beta}_N(\bm X) -\beta) \xrightarrow{D} N\left(0, -\frac{g''(m_*)}{p^2m_*^{2p-2}}\right),
		\end{equation}
where $g(t) := \beta t^p - I(t)$, for $t \in [-1, 1]$, and $m_*=m_*(\beta, p)$ is the unique positive global maximizer of $g$.
\end{thm}

The proof of this result is given in Section \ref{sec:pf_cwmplclt}. Figure \ref{fig:histogram_I} shows the histogram (over $10^6$ replications) of $\sqrt N (\hat{\beta}_N(\bm X)-\beta)$ with $p=4$, $\beta= 0.75$, and $N=20000$. As predicted by the result above, we see a limiting Gaussian distribution, since $\beta= 0.75 >\beta_{\mathrm{CW}}^*(4) \approx 0.689$ is above the estimation threshold.

\begin{figure*}[h]\vspace{-0.15in}
\centering
\begin{minipage}[l]{1.0\textwidth}
\centering
\includegraphics[height=2.65in,width=4.5in]
    {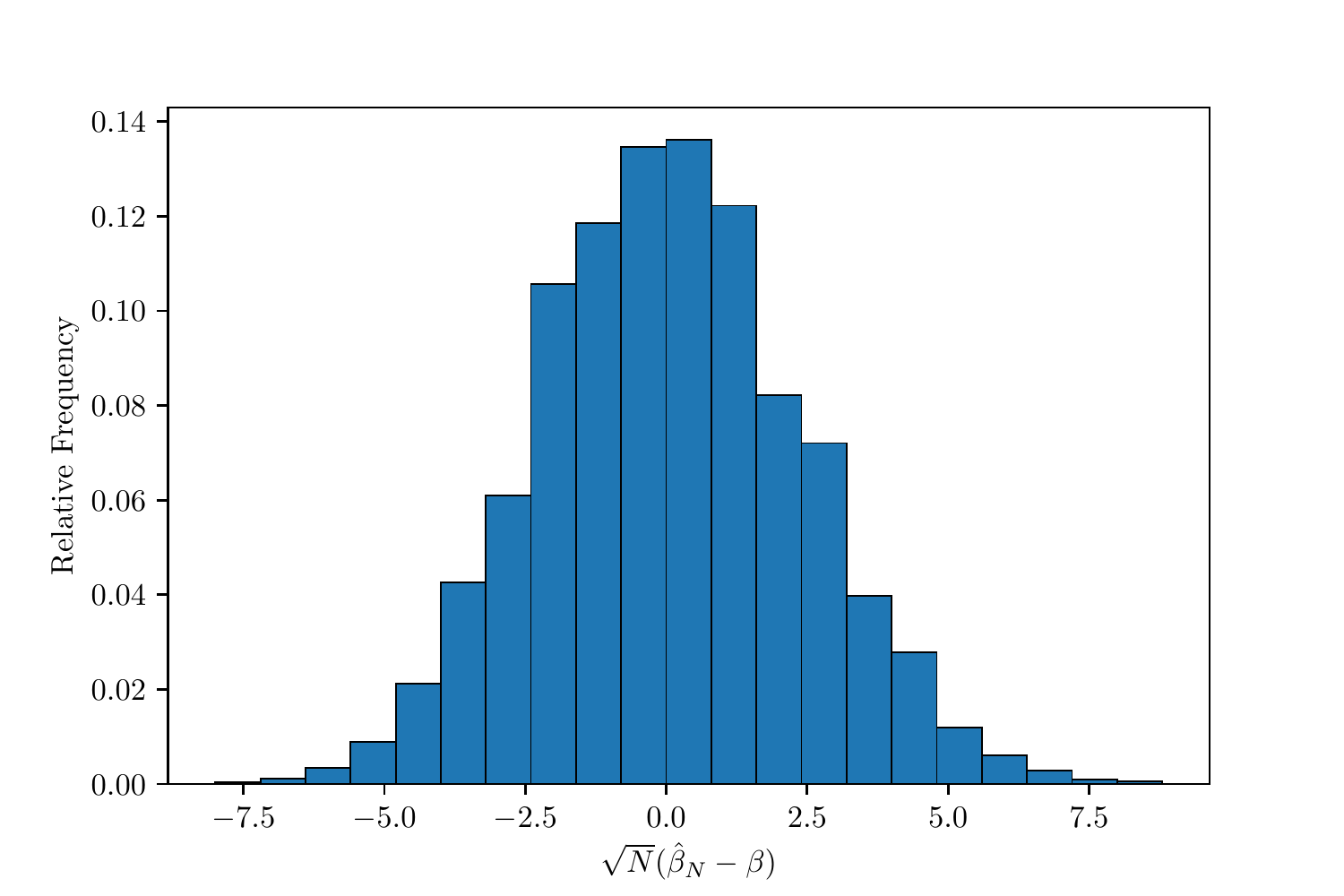}\\
\end{minipage}
\caption{\small{The histogram $\sqrt N (\hat{\beta}_N(\bm X)-\beta)$ in the 4-tensor Curie-Weiss model at $\beta= 0.75 > \beta_{\mathrm{CW}}^*(4) \approx 0.689$ (above the estimation threshold).}}
\label{fig:histogram_I}
\end{figure*}

The result in Theorem \ref{thm:cwmplclt} can be used to construct a confidence interval for the parameter $\beta$ for all points above the estimation threshold. Towards this, note, by \cite[Theorem 2.1]{mlepaper}, that $|\bar X_N| \pto m_*$ under $P_{\beta, p}$, when $\beta > \beta_{\mathrm{CW}}^*(p)$. The result in \eqref{statement1} then implies that 
\begin{equation*}
\left(\hat{\beta}_N(\bm X) - \frac{|\bar X_N|^{1-p}}{p}\sqrt{\frac{-g''(|\bar X_N|)}{N}} z_{1-\frac{\alpha}{2}},~\hat{\beta}_N(\bm X) + \frac{|\bar X_N|^{1-p}}{p}\sqrt{\frac{-g''(|\bar X_N|)}{N}} z_{1-\frac{\alpha}{2}}\right),
\end{equation*}   
is an interval which contains $\beta$ with asymptotic coverage probability $1-\alpha$, whenever $\beta > \beta_{\mathrm{CW}}^*(p)$.\footnote{For $\alpha \in (0, 1)$, $z_\alpha$ is the $\alpha$-th quantile of the standard normal distribution, that is, $\p_\beta(N(0, 1) \leq z_\alpha) = \alpha$.}

\begin{remark}\label{remark:information} {\em (Efficiency of the MPL estimate) An interesting consequence of Theorem \ref{thm:cwmplclt} is that the limiting variance in \eqref{statement1} saturates the Cramer-Rao (information) lower bound of the model, when $\beta > \bw$. To see this, note that the (scaled) Fisher information in the model \eqref{model} (recall that the Cramer-Rao lower bound is the inverse of the Fisher information) is given by,  
$$I_N(\beta) = \frac{1}{N} \e_{\beta}\left[\left(\frac{\mathrm d}{\mathrm d \beta} \log \p_{\beta}(\bs)\right)^2\right] = \mathrm{Var}_\beta (N^\frac{1}{2} \bar{X}_N^p) \rightarrow -\frac{p^2 m_*^{2p-2}}{g''(m_*)},$$ as $N \rightarrow \infty$,  where the last step follows from the asymptotics of $\bar X_N$ derived in \cite{mlepaper}. 
This implies, for $\beta > \bw$, the MPL estimate $\hat \beta_N(\bm X)$ is {\it asymptotically efficient}, which means that no other consistent estimator can have lower asymptotic mean squared error than $\hat \beta_N(\bm X)$ above the estimation threshold. While this has been shown for the maximum likelihood (ML) estimate \cite{comets,mlepaper}, that the MPL estimate, which only maximizes an approximation to the true likelihood, also has this property, is particularly encouraging, as it showcases the effectiveness of the MPL method, both computationally as well as in terms of statistical efficiency. }   
\end{remark}

The results above show that the MPL estimate is $\sqrt N$-consistent and asymptotic efficient whenever $\beta >  \beta_{\mathrm{CW}}^*(p)$. On the other hand, for $\beta <  \beta_{\mathrm{CW}}^*(p)$, we know from Theorem \ref{sbmthr} that consistent estimation is impossible. In particular, this means that  the MPL estimate is inconsistent for $\beta <  \beta_{\mathrm{CW}}^*(p)$. Therefore, the only case that remains is at the threshold  $\beta = \beta_{\mathrm{CW}}^*(p)$. Here, the situation is much more delicate. We address this case in the theorem below, which shows that the MPL is $\sqrt N$-consistent for $p=2$ (with a non-Gaussian limiting distribution), but inconsistent for $p \geq 3$.

\begin{thm}\label{thm:cw_threshold}(Asymptotics of the MPL estimate at the threshold) 
Fix $p \geq 2$ and consider the $p$-spin Curie-Weiss model with interaction tensor $\bm J_N = ((J_{i_1 \ldots i_p}))_{1 \leq i_1, \ldots, i_p \leq N}$, where $J_{i_1  \ldots i_p} = \frac{1}{N^{p-1}}$, for all $1 \leq i_1, \ldots, i_p \leq N$. Suppose $\beta = \beta_{\mathrm{CW}}^*(p)$. Denote by $m_*=m_*(\beta, p) \in (0, 1)$ the unique positive maximizer of the function $g:= \beta t^p - I(t)$, for $t \in [-1, 1]$, and define  
		\begin{align}\label{eq:proportion}
			\alpha := 
			\begin{cases}
				\frac{1}{1+2[(m_*^2-1)g''(m_*)]^{-\frac{1}{2}}} &\quad\text{if}~p~\textrm{is even},\\
				\frac{1}{1+[(m_*^2-1)g''(m_*)]^{-\frac{1}{2}}} &\quad\text{if}~p~\textrm{is odd}.\\
			\end{cases}
		\end{align}
Then, the following hold as $N \rightarrow \infty$, 

\begin{itemize}

\item[(1)]  If $p = 2$ $($recall $\beta_{\mathrm{CW}}^*(2)=\frac{1}{2}$$)$, then for every $t\in \mathbb{R}$, 
\begin{align}\label{eq:threshold_F}
\lim_{N\rightarrow \infty} \p_\beta\left (N^\frac{1}{2}\left(\hat{\beta}_N - \frac{1}{2}\right) \leq t\right) = 
\begin{cases}
F(\sqrt{6t}) - F(-\sqrt{6t}) &\quad\text{if}~t\geq 0\\
0 &\quad\text{if}~t < 0\\
\end{cases}
\end{align}
where $F$ is a probability distribution function with density given by $d F(t) \propto \exp\left(-\frac{t^4}{12}\right) \mathrm dt$.

\item[(2)] If $p \geq 3$, then  
\begin{align}\label{eq:threshold_mple}
\sqrt{N}(\hat{\beta}_N(\bm X) -\beta) \xrightarrow{D} (1-\alpha) N\left(0, -\frac{g''(m_*)}{p^2m_*^{2p-2}}\right) + \alpha \delta_\infty, 
\end{align}
where $g(\cdot)$ is as defined Theorem \ref{thm:cwmplclt} and $\delta_\infty$ denotes the point mass at $\infty$.

\item[(3)] Moreover, at a finer scaling, the following hold: 
\begin{itemize}
			\item[(a)] If $p \geq 4$ is even, then 
\begin{align}\label{eq:threshold_mple_I}			
			N^{1-\frac{p}{2}} \hat{\beta}_N  \xrightarrow{D} \alpha \left(\frac{1}{pZ^{p-2}}\right) + (1-\alpha)\delta_0,
\end{align}
where $Z \sim N(0,1)$. 
			
			\item[(b)] If $p \geq 3$ is odd, then 
\begin{align}\label{eq:threshold_mple_II}			
			N^{1-\frac{p}{2}} \hat{\beta}_N  \xrightarrow{D} \frac{\alpha}{2} \left(\frac{1}{p |Z|^{p-2}}\right) +\frac{\alpha}{2}\delta_\infty + (1-\alpha)\delta_0. 
\end{align}
		\end{itemize}
\end{itemize} 
\end{thm}

The proof of this result is given in Section \ref{sec:pf_cwmplclt}. As in the proof of Theorem \ref{thm:cwmplclt}, the main ingredient in the proof of the above result is the asymptotic distribution of the sample mean at the threshold derived in \cite{comets,mlepaper}. The reason there is a change in the consistency rates of the MPLE as one moves from the 2-spin model to the $p$-spin model, for $p \geq 3$, is because the rate of convergence of the sample mean $\cs$ in the Curie-Weiss model depends on the value of $p$ at the threshold. More precisely, for $p=2$ and $\beta=\beta_{\mathrm{CW}}^*(2)=\frac{1}{2}$, $N^{\frac{1}{4}} \cs \dto F$, where $F$ is as defined in Theorem \ref{thm:cw_threshold} (1) (see \cite[Proposition 4.1]{comets}).  On the other hand, when  $p \geq 3$ and $\beta=\beta_{\mathrm{CW}}^*(p)$, $N^{\frac{1}{2}} \cs$ converges to a mixture of point masses with two or three components depending on whether $p$ is odd or even, respectively (see \cite[Theorem 1.1]{mlepaper}).

\begin{figure*}[h]\vspace{-0.15in}
\centering
\begin{minipage}[l]{0.49\textwidth}
\centering
\includegraphics[width=3.35in]
    {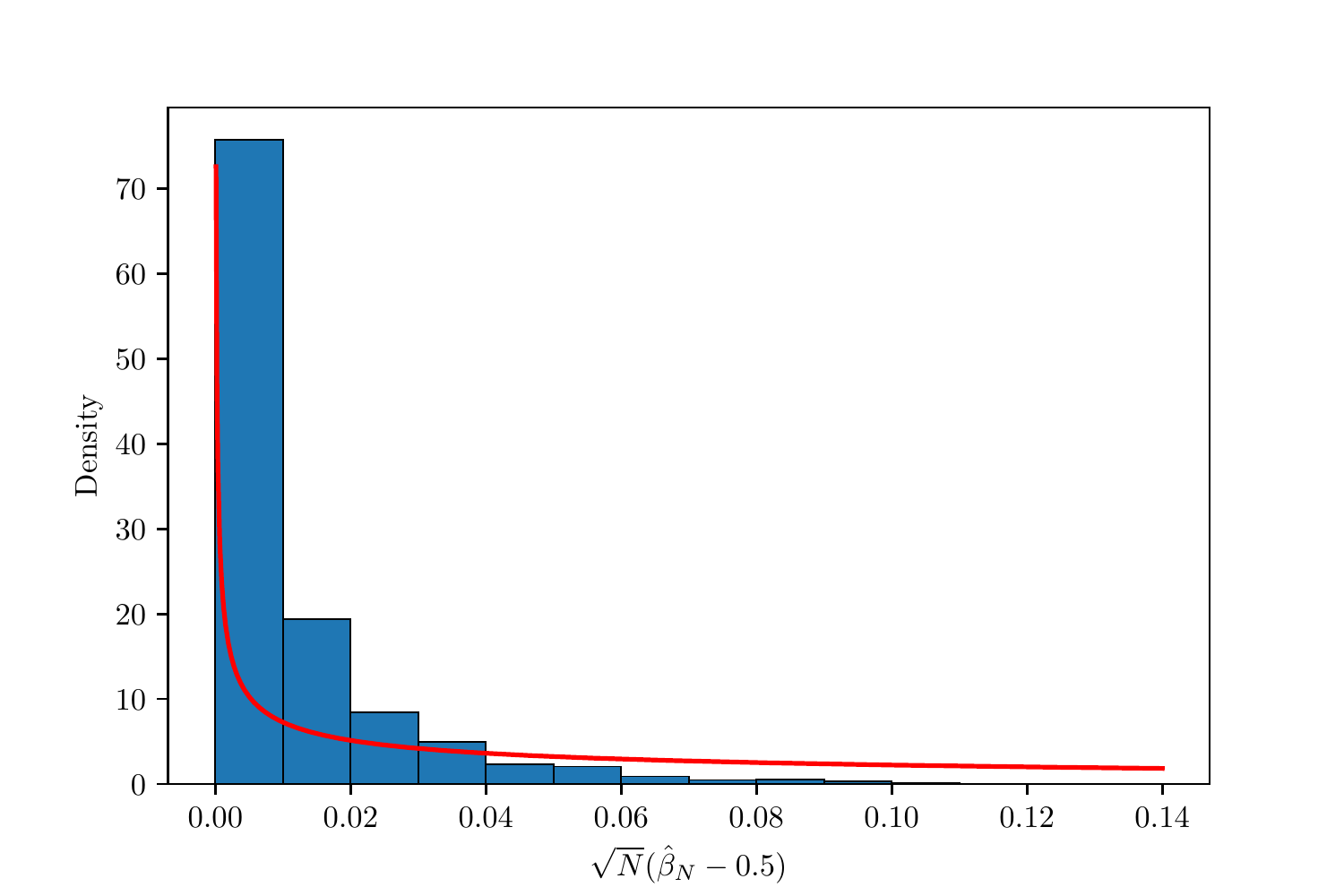}\\
\small{(a)}
\end{minipage} 
\begin{minipage}[l]{0.49\textwidth}
\centering
\includegraphics[width=3.5in]
    {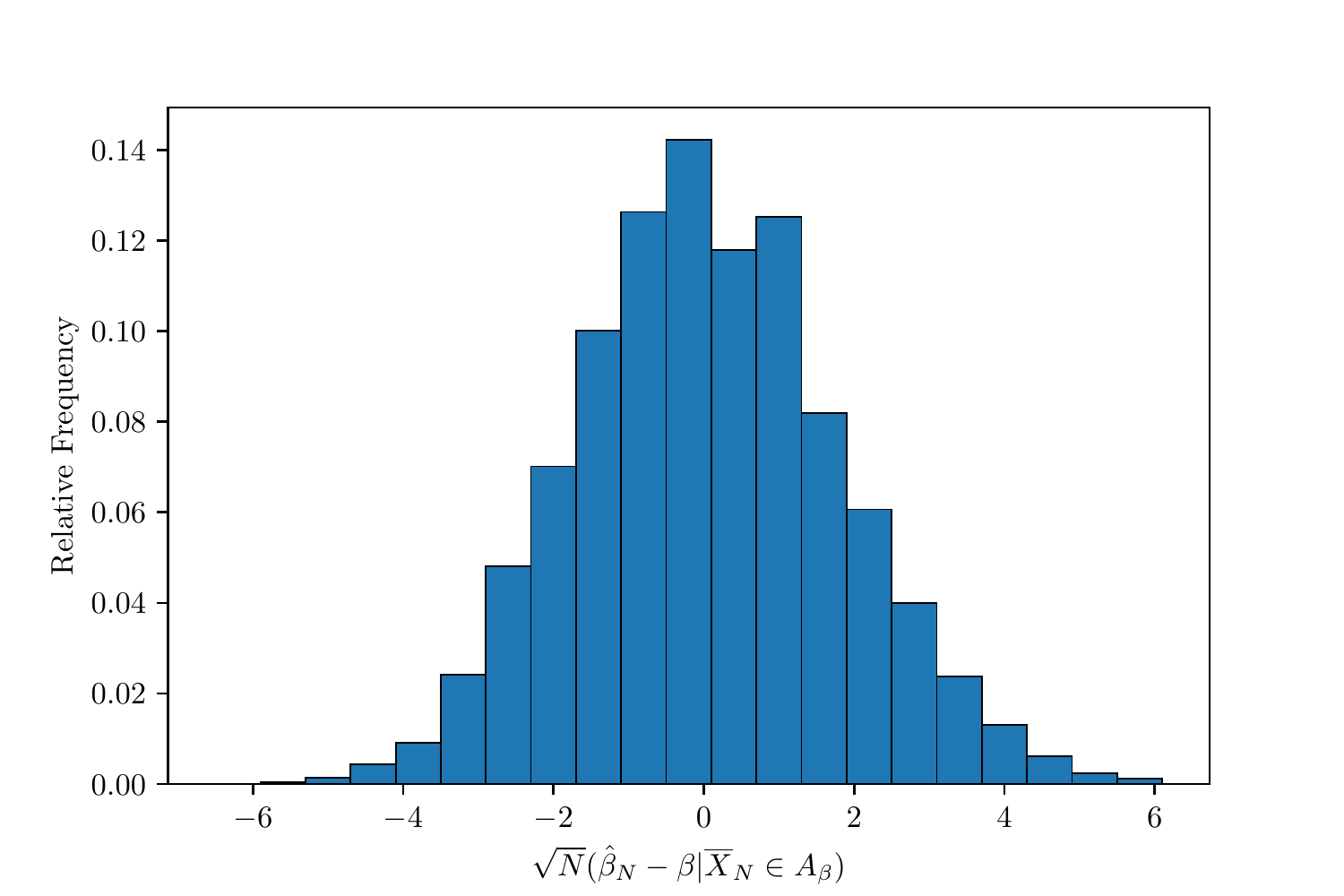}\\ 
    \small{(b)}
\end{minipage} 
\caption{\small{(a) The histogram of  $\sqrt N (\hat{\beta}_N(\bm X)-\beta)$  in the 2-tensor Curie-Weiss model at the estimation threshold ($\beta=\frac{1}{2}$) and the limiting density function (in red); and (b) the histogram of the conditional distribution $\sqrt N (\hat{\beta}_N(\bm X)-\beta)|\{\cs \in A_\beta\}$, where $A_\beta$ is the interval $[-1, 1]$ minus a small neighborhood around zero, in the 4-tensor Curie-Weiss model at the estimation threshold, which has a limiting normal distribution.}}
\label{fig:histogram_II}
\end{figure*}

Taking derivatives in \eqref{eq:threshold_F} shows that for $p=2$ the MPL estimate has a limiting Gamma distribution with density $g(a) \propto \frac{1}{\sqrt a }e^{-3 a^2} \mathrm da$.  Figure \ref{fig:histogram_II}(a) shows the histogram of $\sqrt N (\hat{\beta}_N(\bm X)-\beta)$ for $p=2$ and $\beta= \beta_{\mathrm{CW}}^*(2)= 0.5$, and the limiting density function (plotted in red). On the other hand, for $p \geq 3$, Theorem \ref{thm:cw_threshold} (3)
shows that the MPL estimate is inconsistent at the threshold (in fact, $\hat \beta_N(\bm X) \pto \infty$, for $p \geq 3$ and  $\beta=\beta_{\mathrm{CW}}^*(p)$). However, even though for $p \geq 3$ the MPL estimate is inconsistent when $\beta=\beta_{\mathrm{CW}}^*(p)$, Theorem \ref{thm:cw_threshold} (2) shows $\sqrt N (\hat{\beta}_N(\bm X)-\beta)$ has a Gaussian limit with probability $1-\alpha$,  that is, MPL estimate is $\sqrt N$-consistent at this point with probability $1-\alpha$. In fact, the proof of Theorem \ref{thm:cw_threshold} (2) shows that $\hat{\beta}_N(\bm X)$ is  not $\sqrt N$-consistent at the threshold for $p \geq 3$, only when $\cs$ is close to zero. More precisely, the proof shows that $\sqrt N (\hat{\beta}_N(\bm X)-\beta)|\{\cs \in A_\beta \} \dto N(0, -\frac{g''(m_*)}{p^2m_*^{2p-2}})$, if $A_\beta=[-1, 1] \backslash B_0$, where $B_0$ is a small neighborhood of zero. This is illustrated in Figure \ref{fig:histogram_II}(b) which plots the histogram of this conditional distribution for $p=4$ and $\beta=0.6888 \approx \beta_{\mathrm{CW}}^*(4)$ .

\subsection{Organization}
The rest of the paper is organized as follows. In Section \ref{sec:3} we prove Theorem \ref{sec:3}. The proofs of Corollary \ref{skthreshold} and Corollary \ref{boundeddeg} are given in Section \ref{skproof} and Section \ref{sec:boundeddegpf}, respectively. The proof of Theorem \ref{sbmthr} is given in Section \ref{proof6}. The proofs of Theorem \ref{thm:cwmplclt} and Theorem \ref{thm:cw_threshold} are given in Section \ref{sec:pfcwclt}. Additional properties of the Curie-Weiss threshold are given in Appendix \ref{cwtp}.

\section{Proof of Theorem \ref{chextension}}\label{sec:3}

In this section, we prove of Theorem \ref{chextension}.  We first state the two main technical estimates required in the proof, and show how these results can be used to complete the proof of Theorem \ref{chextension}. As mentioned before, the first step in the proof of Theorem \ref{chextension}  is to show that the (scaled) log-pseudolikelihood concentrates around zero at the true parameter value $\beta > 0$ at the desired rate. This is achieved by proving the following second-moment estimate on the scaled log-pseudolikelihood function. The proof of this lemma is given in Section \ref{sec:lem1}.

\begin{lem}\label{lm:boundsecondmoment} Let $\beta > 0$ be such that assumption $(1)$ of Theorem \ref{chextension} holds. Then 
	$$\e_\beta\left[s_\bs^2(\beta)\right] = O_{\beta, p}\left(\frac{1}{N}\right),$$
where $s_{\bm X}(b) = \frac{1}{pN} \frac{\partial \log L(b|\bm X)}{\partial b} = \frac{1}{N} (H_N(\bm X) -  \sum_{i=1}^N m_i(\bm X) \tanh (p b m_i(\bs) ) )$. 
\end{lem}

The next step of the proof is to show the strong concavity of the log-pseudolikelihood, that is, $-\frac{\partial}{\partial \beta}s_{\bm X}(\beta ) $ is strictly positive and bounded away from $0$ with high probability. To this end, note that for any $M > 0$, 
\begin{align}\label{eq:sbeta_II}
-\frac{\partial}{\partial \beta}s_{\bm X}(\beta ) & =\frac{p}{N}\sum_{i=1}^N m_i(\bm X)^2\text{sech}^2( p \beta m_i(\bm X)) \nonumber \\ 
& \geq  \frac{p}{N}\text{sech}^2(p \beta M)\sum_{i=1}^N m_i(\bm X)^2 \bm  1\{|m_i(\bm X)|\leq M\}.
\end{align}
Therefore, to show that $-\frac{\partial}{\partial \beta}s_{\bm X}(\beta )$ is strictly positive, it suffices to show that $\sum_{i=1}^N m_i(\bm X)^2 \bm  1\{|m_i(\bm X)|\leq M\}$ is $\Omega(N)$ with high probability. This is formalized in the following lemma which is proved in Section \ref{sec:lem2}.

\begin{lem}\label{derivative} Fix $0 < \delta < 1$. Then under the assumptions in Theorem \ref{chextension},  there exists $\varepsilon=\varepsilon(\delta, \beta)>0$ and $M=M(\delta, \beta)<\infty$  such that 
	\begin{align*}
	\P_{\beta}\left(\sum_{i=1}^N m_i(\bm X)^2\bm 1\{|m_i(\bm X)|\leq M\}\geq  \varepsilon N \right)\geq  1-\delta,
	\end{align*}
	for all $N$ large enough. 
\end{lem}

The proof of Theorem \ref{chextension} can now be easily completed using the above lemmas. To this end, note that for any $M_1 > 0$, by Chebyshev's inequality and Lemma \ref{lm:boundsecondmoment} we have, 
\begin{align*}
\P_{\beta}\left(|s_{\bm X}(\beta )|> \frac{M_1}{\sqrt N}\right)\lesssim_{\beta,p} \frac{1}{M_1^2}.
\end{align*}
Now, fix $\delta>0$. Therefore, it is possible to choose $M_1=M_1(\delta, \beta)$ such that the RHS above is less than $\delta$. Next, by Lemma \ref{derivative} there exists $\varepsilon=\varepsilon(\delta, \beta)>0$ and $M_2=M_2(\varepsilon,\delta, \beta)<\infty$  such that 
\begin{align*}
\P_{\beta}\left(\sum_{i=1}^N m_i(\bm X)^2\bm 1\{|m_i(\bm X)|\leq M_2\}\geq  \varepsilon N \right)\geq  1-\delta, 
\end{align*}
for $N$ large enough. Thus, defining 
$$T_N:=\left\{\bm X \in \cC_N: |s_{\bm X}(\beta )| \leq \frac{M_1}{\sqrt N},~ \sum_{i=1}^N m_i(\bm X)^2 \bm  1\{|m_i(\bm X)|\leq M_2\}\geq  \varepsilon N \right\},$$ gives $\P_{\beta}(T_N)\geq  1-2\delta$, for $N$ large enough. For $\bm X \in T_N$, recalling \eqref{eq:sbeta_II}, gives 
\begin{align*}
-\frac{\partial}{\partial \beta}s_{\bm X}(\beta ) \geq  \frac{p}{N}\text{sech}^2(p \beta M_2)\sum_{i=1}^N m_i(\bm X)^2 \bm  1\{|m_i(\bm X)|\leq M_2\} &\geq p \varepsilon \text{sech}^2(p \beta M_2). 
\end{align*}			
Therefore, for $\bm X \in T_N$, 
\begin{align*}
\frac{M_1}{\sqrt N}\geq  |s_{\bm X}(\beta)|=|s_{\bm X}(\beta)-s_{\bm X}(\hat \beta_N(\bm X))| &  \geq   - \int_{\beta\wedge \hat \beta_N(\bm X)}^{\beta\vee \hat \beta_N(\bm X)} \frac{\partial}{\partial \beta}s_{\bm X}(\beta ) \mathrm d\beta \nonumber\\
& \geq \frac{\varepsilon }{M_2 }|\tanh(p M_2\hat \beta_N(\bm X))-\tanh(p M_2\beta)| . 
\end{align*}
Then, defining $M=M(\delta, \beta):=\frac{M_2}{M_1\varepsilon}$, shows that 
$$\P_{\beta}\left(\sqrt{N}|\tanh(p M_2\hat \beta_N(\bm X))-\tanh (p M_2\beta)|\leq R \right)\geq 1- 2\delta.$$
The proof of Theorem \ref{chextension} now follows by inverting the $\tanh$ function.

\subsection{Proof of Lemma \ref{lm:boundsecondmoment}}\label{sec:lem1}

For $\bm x, \bm x' \in \cC_N$ define 
\begin{equation*}
F(\bt,\bt')=\frac{1}{2}\sum_{i=1}^N \left(m_i(\bt)+m_i(\bt')\right)(x_i-x_i'),
\end{equation*}
where $m_i$ is as defined in \eqref{eq:conditional}. Note that $F$ is antisymmetric, that is,  $F(\bt, \bt') = -F(\bt',\bt)$.

Now, choose a coordinate $I \in \{1, 2, \ldots, N\}$ uniformly at random and replace the $I$-th coordinate of $\bm X \sim \P_{\beta}$ by a sample drawn from the conditional distribution of $X_u$ given $(X_v)_{v \ne I}$. Denote the resulting vector by $\bm X'$. Note that $F(\bs,\bs')=m_I(\bs)(X_I-X_I')$. Then 
\begin{align}\label{eq:sx_II}
f(\bs):=\mathbb{E}_\beta\left(F(\bs,\bs')|\bs\right) &= \frac{1}{N}\sum_{i=1}^N m_i(\bs) \left\{ X_i- \e_{\beta}\left(X_i|(X_j)_{j\neq i}\right) \right\} \nonumber \\ 
&=\frac{1}{N}\sum_{i=1}^N m_i(\bs)\left(X_i-\tanh\big(p \beta m_i(\bs) \big)\right) \nonumber \\
&=s_\bs(\beta).
\end{align}
Now, since $(\bs,\bs')$ is an exchangeable pair, 
\begin{equation*}
\mathbb{E}_\beta\left(f(\bs)F(\bs,\bs')\right)=\mathbb{E}_\beta\left(f(\bs')F(\bs',\bs)\right).
\end{equation*}
Again, because $F$ is antisymmetric, we have $\mathbb{E}_\beta\left(f(\bs')F(\bs',\bs)\right)=-\mathbb{E}_\beta\left(f(\bs')F(\bs,\bs')\right).$ Hence,  
\begin{align}
\mathbb{E}_\beta\left(f(\bs)^2\right)=\e_{\beta}\left(f(\bs)\e_{\beta}\left[F(\bs,\bs')|\bs\right)\right] & = \mathbb{E}_\beta\left(f(\bs)F(\bs,\bs')\right) \nonumber \\ 	 		
\label{expectation f squared} &=\tfrac{1}{2}\mathbb{E}_\beta\left((f(\bs)-f(\bs'))F(\bs,\bs')\right).
\end{align}

Now, for any $1 \leq t \leq N$ and $\bm x \in \cC_N$, let 
$$\bm x^{(t)}=(x_1, x_2, \ldots, x_{t-1}, 1- x_t, x_{t+1}, \ldots, x_N),$$ 
and 
\begin{equation*}
p_t(\bt):=\mathbb{P}_\beta(X_t'=-x_t|\bs=\bt,I=t)=\frac{e^{-p\beta x_t m_t(\bt)}}{e^{-p\beta  m_t(\bt)} + e^{p\beta m_t(\bt)}}.
\end{equation*}
This implies, 		
\begin{align}\label{cond2step} 
\E_{\beta}((f( \bm X)-f(\bm X'))F( \bm X,  \bm X')| \bm X) \nonumber &=\frac{1}{N}  \sum_{t=1}^N (f( \bm X)-f( \bm X^{(t)})) F( \bm X,  \bm X^{(t)}) p_t( \bm X) \nonumber \\ 	
&=\frac{2}{N}\sum_{t=1}^N m_t(\bs)X_t p_t(\bs) (f(\bs)-f(\bm X^{(t)})).
\end{align} 
For $1 \leq s, t \leq N$, let $a_s(\bt):=x_s -\tanh(p \beta m_s(\bt))$ and $b_{s t}(\bt):=\tanh( p \beta m_s(\bt))-\tanh(p \beta m_s(\bt^{(t)}))$. Then, noting that $f(\bt)= \frac{1}{N} \sum_{s=1}^N m_s(\bt) a_s(\bt)$ gives 
\begin{align} \label{fxx} 
f(\bs)-f(\bm X^{(t)}) &=\frac{1}{N}\sum_{s=1}^N (m_s(\bs)-m_s(\bm X^{(t)}))a_s(\bs)+\frac{1}{N}\sum_{s=1}^N m_s(\bm X^{(t)}) (a_s(\bs)-a_s(\bm X^{(t)}) ) \nonumber  \\ 
&= A_t + B_t + C_t,
\end{align} 
where $A_t := \frac{2(p-1)X_t}{N}\sum_{s=1}^N  J_{s t}(\bs) a_{s}(\bs)$, $B_t:= \frac{2 m_t(\bs)X_t}{N}$, and $C_t :=-\frac{1}{N}\sum_{s=1}^Nm_s(\bm X^{(t)}) b_{s t}(\bs)$. Then using \eqref{expectation f squared} and \eqref{cond2step}, we have 
\begin{equation*} 
\mathbb{E}_\beta\left(f(\bs)^2\right)= \frac{1}{N}\sum_{t=1}^N\e_{\beta}\left[\left(A_t + B_t + C_t \right)m_t(\bs)X_t p_t(\bs)\right]. 
\end{equation*}

Now, define the following three quantities: $$\bm a(\bs) := \left(a_1(\bs),\ldots,a_N(\bs)\right), \quad \bm m (\bs) := \left(m_1(\bs),\ldots,m_N(\bs)\right),$$ and $\bm M(\bs) := \left(m_1(\bs)p_1(\bs),\ldots,m_N(\bs)p_N(\bs)\right)$. Note that $\bm m(\bs) = \bs \bm J_N(\bs)^\top$. Also, observe that each entry of $\bm a(\bs)$ is bounded in absolute value by $2$, hence, $\|\bm a(\bs)\| \leq 2\sqrt{N}$. Moreover, using $p_t(\bm X) \leq 1$, 
\begin{align}\label{eq:norm_bound}
\|\bm M(\bs)\| = \left(\sum_{i=1}^N m_i(\bs)^2 p_i(\bs)^2\right)^\frac{1}{2} \leq \|\bm m(\bs)\| = \|\bm J_N(\bs)\bs^\top\|\leq \sqrt N \|\bm J_N(\bs)\| . 
\end{align}
Hence, recalling the definition of $A_t$ from \eqref{fxx} gives 
\begin{align}\label{firstterm}
\Bigg|\frac{1}{N}\sum_{t=1}^N A_t m_t(\bs) X_t p_t(\bs)\Bigg| & = \frac{2(p-1)}{N^2} \Big|\bm a(\bs) \bm J_N(\bs) \bm M(\bs)^\top\Big|\nonumber\\
&\lesssim_p \frac{1}{N^2}\|\bm J_N(\bs)\|\|\bm a(\bs)\|\|\bm M(\bs)\|\nonumber\\
&\lesssim_p \frac{||\bm J_N(\bm X)||^2}{N} .
\end{align}
Next, we  consider the term corresponding to $B_t$: 
\begin{align}\label{secondterm}
\Bigg|\frac{1}{N}\sum_{t=1}^N B_t m_t(\bs) X_t p_t(\bs)\Bigg| = \frac{2}{N^2}\Bigg|\sum_{t=1}^N m_t^2(\bs)p_t(\bs)\Bigg| \lesssim \frac{1}{N^2} \|\bm m (\bs)\|^2 \leq \frac{||\bm J_N(\bm X)||^2}{N},
\end{align}
where the last step uses \eqref{eq:norm_bound}.

Finally, we consider the term corresponding to $C_t$. Let us define the matrix $\bm J_{N, 2}(\bs) := ((J_{ij}(\bs)^2))_{1\leq i,j\leq N}$. Then, denoting by $\bm e_i$ the vector in $\R^N$ with the $i$-th entry 1 and 0 everywhere else, we get 
\begin{equation*}
\|\bm J_{N, 2}(\bs)\| \leq \max_{1 \leq i \leq N} \sum_{j=1}^N J_{ij}(\bs)^2 \leq \max_{1 \leq i \leq N}
\|\bm e_{i}^\top \bm J_N(\bs)\|^2 \leq  \|\bm J_N(\bs)\|^2.
\end{equation*}
Let $h(x):= \tanh\left(p \beta x \right)$. It is easy to check that $\|h''\|_\infty \leq \beta^2$. Hence, by a Taylor expansion, for $1 \leq s \leq N$, 
\begin{align}\label{taylorh}
\big|h(m_s(\bs))- h(m_s(\bm X^{(t)})) - (m_s(\bs) - m_s(\bm X^{(t)}) ) h'(m_s(\bs))\big| \lesssim_{\beta} \left(m_s(\bs) - m_s(\bm X^{(t)})\right)^2.
\end{align}
Note that $m_s(\bs)-m_s(\bm X^{(t)})=2(p-1)J_{ s t}(\bs)X_t$ and $h(m_s(\bs))- h(m_s(\bm X^{(t)})) = \tanh( p \beta m_s(\bs))-\tanh(p \beta m_s(\bs^{(t)})) = b_{s t}(\bs)$, hence, \eqref{taylorh} can be rewritten as:
\begin{equation*}
\left|b_{s t}(\bs)-2(p-1)J_{s t}(\bs)X_t h'(m_s(\bs))\right|\lesssim_{\beta, p} J_{s t}(\bs)^2.
\end{equation*} 
Using the above bounds, we have, for any $\bm x=(x_1, x_2, \ldots, x_N) ,\bm y = (y_1, y_2, \ldots, y_N) \in\mathbb{R}^N$, 
\begin{align}
& \left|\sum_{1 \leq s, t \leq N} x_s y_t b_{s t}(\bs)\right| \nonumber \\ 
&\leq \left|\sum_{1 \leq s, t \leq N}2(p-1)x_s y_t J_{s t}(\bs)X_t h'(m_s(\bs))\right|+\left|\sum_{1 \leq s, t \leq N}x_s y_t \big(b_{s t}(\bs)-2(p-1)J_{s t}(\bs)X_t h'(m_s(\bs))\big)\right| \nonumber \\
\nonumber & \lesssim_{\beta, p}  \|\bm J_N(\bs)\| \left(\sum_{s=1}^N (x_s h'(m_s(\bs)))^2 \right)^{\frac{1}{2}} \left(\sum_{t=1}^N (y_t X_t)^2 \right)^{\frac{1}{2}} + \sum_{1 \leq s, t \leq N}|x_s y_t|\ J_{s t}(\bs)^2\\
\nonumber &\lesssim_{\beta, p} \|\bm J_N(\bs)\| \|\bm x\| \|\bm y\| +  \|\bm J_{N, 2}(\bs)\| \|\bm x\|\|\bm y\| \tag*{(using $h'(m_s(\bs)) \lesssim_{\beta, p} 1$)}\\
\label{ints1} &\lesssim_{\beta, p}  \left(  \|\bm J_N(\bs)\|  +  \|\bm J_N(\bs)\|^2 \right)\|\bm x\|\|\bm y\|.
\end{align}
Again, by a Taylor expansion and using the bound $||h'||_{\infty} \lesssim_{\beta, p} 1$, gives 
$$|b_{s t}(\bs)| \lesssim_{\beta, p} |m_s(\bs)-m_s(\bm X^{(t)}) | \lesssim_{\beta, p} |J_{s t}(\bs)|.$$ Consequently,
\begin{equation}\label{ints2}
\left|\sum_{1 \leq s, t  \leq N}x_s y_t  J_{s t}(\bs)b_{s t}(\bs)\right| \lesssim_{\beta, p} \sum_{1 \leq s, t  \leq N}|x_s y_t | {J}_{s t}(\bs)^2 \leq ||\bm J_N(\bm X)||^2 \|\bm x\|\| \bm y\|.
\end{equation}
Now, recalling the definition of $C_t$ from \eqref{fxx} gives, 
\begin{align}\label{thirdterm}
\nonumber\left|\frac{1}{N}\sum_{t=1}^N C_t m_t(\bs) X_t p_t(\bm X)\right|&=\left|\frac{1}{N^2} \sum_{1 \leq s, t \leq N} m_s( \bm X^{(t)} )b_{s t}(\bs)m_t(\bs) X_t p_t(\bs) \right|\\
\nonumber 
&=\left|\frac{1}{N^2}\sum_{1 \leq s, t \leq N} \big(m_s(\bs)-2(p-1)J_{s t}(\bs) X_t \big) b_{s t}(\bs) m_t(\bs) X_t p_t(\bs)\right|\\
& \lesssim_{\beta, p} \frac{||\bm J_N(\bm X)||^3 + ||\bm J_N(\bm X)||^4}{N}, 
\end{align}
where the last step uses \eqref{ints1}, \eqref{ints2}, and $\|\bm m(\bs)\| = \|\bm J_N(\bs)\bs^\top\|\leq \sqrt N \|\bm J_N(\bs)\|$.

Combining \eqref{firstterm}, \eqref{secondterm} and \eqref{thirdterm}, it follows that	$\mathbb{E}_\beta(f(\bm X)^2) \lesssim_{\beta, p} \frac{1}{N}$, since by condition of $(1)$ of Theorem \ref{chextension}, $\E_{\beta}(||\bm J_N(\bm X)||^4 )$ is uniformly bounded. This completes the proof of the lemma, since recalling \eqref{eq:sx_II}, $f(\bs) = s_\bs(\beta)$. \qed

\subsection{Proof of Lemma \ref{derivative}}\label{sec:lem2}

We begin with the following simple observation, which says that if $\liminf_{N\rightarrow\infty}\frac{1}{N}{F_N(\beta)} > 0$, we can find a $\gamma$ small enough such that $\liminf_{N\rightarrow\infty}\frac{1}{N}{F_N'(\beta-\gamma)} > 0$.

\begin{obs}\label{obs:Fderivative} Suppose $\beta > 0$ is such that $\liminf_{N\rightarrow\infty} \frac{1}{N} F_{N}(\beta) > 0$. Then $$\lim_{\delta\rightarrow 0}\liminf_{N\rightarrow\infty}\frac{1}{N} F_{N}(\beta-\delta) > 0.$$
\end{obs}

\begin{proof} Denote $K:=\sup_{N \geq 1}  \E_\beta (||J_N(\bm X)||) < \infty $ Then, $$F_{N}'(\beta): = \frac{\mathrm d}{\mathrm d \beta} F_{N}'(\beta) = \E_\beta (H_N(\bm X)) = \E_\beta (\bm X' \bm J_N(\bm X) \bm X ) \leq N  \E_\beta (|| \bm J_N(\bm X)||) \leq K N.$$ 
	Therefore, by a Taylor expansion
	$$\lim_{\delta\rightarrow 0}\liminf_{N\rightarrow\infty}\frac{1}{N} F_{N}(\beta-\delta)  \geq  \lim_{\delta\rightarrow 0}\liminf_{N\rightarrow\infty}\left(\frac{1}{N} F_{N}(\beta) -   \delta  K \right)>0,$$
	as required. 
\end{proof}

Now, note that for any $\varepsilon, \gamma >0$,  
\begin{align*}
&\P_{\beta}(H_N(\bm X)<\varepsilon N)=\P_{\beta}(e^{-\gamma H_N(\bm X)}>e^{-\gamma \varepsilon N}) \leq e^{\gamma \varepsilon N+F_N(\beta-\gamma)-F_N(\beta)}
\end{align*}
which, on taking logarithms, implies that 
\begin{align*}
\log \P_{\beta}(H_N(\bm X)<\varepsilon N)\leq  \varepsilon\gamma N - \int_{\beta-\gamma}^{\beta}F_N'(t)\mathrm dt\leq  \varepsilon \gamma N -F_N'(\beta-\gamma)\gamma ,
\end{align*} 
by the monotonicity of $F_N'(\cdot)$. Dividing both sides by $N$ and taking limits as $N\rightarrow \infty$ followed by $\varepsilon\rightarrow 0$ we have
$$\lim_{\varepsilon\rightarrow0}\limsup_{N\rightarrow\infty}\frac{1}{N}\log \P_{\beta}(H_N(\bm X)<\varepsilon N)\leq -\liminf_{N\rightarrow\infty}\frac{1}{N}{F_N'(\beta-\gamma)} \leq -\liminf_{N\rightarrow\infty}\frac{F_N(\beta-\gamma)}{N(\beta-\gamma)} <0,$$
by choosing $\gamma$ small enough (by Observation \ref{obs:Fderivative}).  This shows that, for every $0 < \delta < 1$ there exists  $\varepsilon=\varepsilon(\delta)>0$ such that, for $N$ large enough,   
\begin{equation}
\P_{\beta}(H_N(\bm X)< 2\varepsilon N )\leq \delta. 
\label{c00}
\end{equation}

Next, by Lemma \ref{lm:boundsecondmoment} and Chebyshev's inequality, there exists  $M_1=M_1(\delta)< \infty$ such that 
\begin{eqnarray}
\P_{\beta}\left(|s_{\bm X}(\beta )|> \frac{M_1}{\sqrt N} \right)  \leq \delta.
\label{c1}
\end{eqnarray}
Moreover, note that for any $M_2 > 0$, 
$$\sum_{i=1}^N|m_i(\bm X)|\bm 1\{|m_i(\bm X)|> M_2\} \leq \frac{1}{M_2} \sum_{i=1}^N m_i(\bm X)^2 =\frac{1}{M_2}  ||J_N(\bm X) \bm X^\top||^2 \leq \frac{N ||J_N(\bm X)||^2}{M_2}.$$
Therefore, using Markov's inequality and condition $(1)$ of Theorem \ref{chextension}, we can choose $M_2=M_2(\delta) < \infty$ such that for all $N$ large enough, 
\begin{align}\label{eq:msigma_epsilon}
\P_{\beta}\left(\sum_{i=1}^N|m_i(\bm X)|\bm 1\{|m_i(\bm X)|> M_2\} > \varepsilon N \right) \leq \frac{\E_\beta(||J_N(\bm X)||^2)}{ \varepsilon M_2} \leq \delta. 
\end{align} 

Then, defining 
\begin{align*}
T_{N}:= \left\{\bm X\in \cC_N: H_N(\bm X)\geq  2\varepsilon N, ~  |s_{\bm X}(\beta )| \leq \frac{M_1}{\sqrt N}, ~ \sum_{i=1}^N|m_i(\bm X)|\bm 1\{|m_i(\bm X)|> M_2\} \leq \varepsilon N   \right \}, 
\end{align*}
and combining (\ref{c00}), (\ref{c1}), and~(\ref{eq:msigma_epsilon}), gives $\P_{\beta}(T_N )\geq  1-3\delta$, for $N$ large enough. Now, for $\bm X \in T_N$, 
\begin{align*}
& \sum_{i=1}^N m_i(\bm X)^2 \bm  1\{|m_i(\bm X)|\leq M_2\}+\varepsilon N \\ 
& \geq  \frac{1}{p \beta} \sum_{i=1}^N |m_i(\bm X)| \tanh(p \beta |m_i(\bm X)|) \bm  1\{|m_i(\bm X)|\leq M_2\} + \sum_{i=1}^N |m_i(\bm X)| \bm  1\{|m_i(\bm X)| >  M_2\}  \tag*{(using $\tanh x \leq x$)} \\ 
& \gtrsim_{p,\beta} \sum_{i=1}^N |m_i(\bm X)| \tanh(p \beta |m_i(\bm X)|) \bm  1\{|m_i(\bm X)|\leq M_2\} + \sum_{i=1}^N |m_i(\bm X)|  \tanh(p \beta |m_i(\bm X)| ) \bm  1\{|m_i(\bm X)| > M_2\} \tag*{(using $\tanh x \leq 1$)}   \\ 
& \geq \sum_{i=1}^N m_i(\bm X)\tanh(p \beta m_i(\bm X))\\ 
& = H_N(\bm X)-Ns_{\bm X}(\beta ) \geq 2\varepsilon N - M_1\sqrt{N}. 
\end{align*}
Thus, on the set $T_N$,
$$\sum_{i=1}^N m_i(\bm X)^2 \bm  1\{|m_i(\bm X)|\leq M_2\} \gtrsim_{p,\beta} 2\varepsilon N- M_ 1\sqrt{N}> \varepsilon N,$$ for all $N$ large enough. This completes the proof of Lemma \ref{derivative}. \qed

\section{Proof of Corollary \ref{skthreshold}}\label{skproof}

To prove Corollary \ref{skthreshold} we will verify that the conditions in Theorem \ref{chextension} hold with probability 1. As mentioned before, in this case condition (2) is easy to verify. To this end, note that by \cite[Theorem 1.1]{bovier}, $\lim_{N \rightarrow \infty} \frac{1}{N} F_N(\beta) = \beta^2 /2$ almost surely, for $\beta > 0$ small enough. This implies, since $F_N$ on increasing on the positive half-line, $\lim_{N \rightarrow \infty} \frac{1}{N} F_N(\beta) >0$ almost surely, for all $\beta > 0$. This establishes condition (2) in Theorem \ref{chextension}.

We now proceed to verify condition (1). To begin with, fix $\bt \in \sa_N$ and consider the Gaussian process 
\begin{align}\label{eq:sphere_ux}
G_{\bm u}(\bt):= \bm u^\top \bm J_N(\bt) \bm u
\end{align} 
indexed by $\bm u \in S^{N-1} := \{\bm t \in \mathbb{R}^N: \|\bm t\| = 1\}$. Here, $\bm J_N(\bt)$ is the local interaction matrix corresponding to the tensor \eqref{eq:JN_sk} of the $p$-tensor SK model. Note that the maximum eigenvalue of $\bm J_N(\bt)$ can be expressed as $\lambda_{\max}\left(\bm J_N(\bt)\right) = \sup_{\bm u \in S^{N-1}} G_{\bm u}(\bt)$.\footnote{For any $N \times N$  matrix  $\bm A$, $\lambda_{\max}(\bm A)$ and $\lambda_{\min}(\bm A)$ denotes the maximum and the minimum eigenvalue of $\bm A$, respectively.}

\begin{lem}\label{tec2} Fix $\bm x \in \cC_N$ and consider the Gaussian process $\{G_{\bm u}(\bm x) : \bm u \in S^{N-1} \}$ as defined above in \eqref{eq:sphere_ux}. Then, the following hold: 
	\begin{itemize}
		\item[(1)]~For every vector $\bm u \in S^{N-1}$, $\e\left[G_{\bm u}(\bm x)^2\right] \lesssim_p \frac{1}{N}$;
		\item[(2)]~For vectors $\bm u, \bm v \in S^{N-1}$, $\e\left[G_{\bm u}(\bm x) - G_{\bm v}(\bm x)\right]^2 \lesssim_p \frac{1}{N} \sum_{i=1}^N (u_i-v_i)^2$. 
	\end{itemize}
\end{lem}

\noindent\textit{Proof of} $(1)$: Fix $\bm x \in \cC_N$ and $\bm u \in S^{N-1}$. Then  	
\begin{align*}
G_{\bm u}(\bt) &= \sum_{1 \leq i_1,\ldots,i_p \leq N} J_{i_1\ldots i_p} u_{i_1}u_{i_2}x_{i_3}\ldots x_{i_p}\\ &= (p-2)!\sum_{1\leq i_1<\ldots<i_p \leq N} J_{i_1\ldots i_p} \left(\sum_{1\leq s \neq t \leq p} u_{i_s} u_{i_t} \prod_{a\in \{1, 2, \ldots, p\}\setminus\{s,t\}} x_{i_a}\right).
\end{align*} 
Hence, 
\begin{align}
\e \left[G_{\bm u} (\bt)^2\right] &\lesssim_p \frac{1}{N^{p-1}}\sum_{1\leq i_1<\ldots<i_p \leq N} \left(\sum_{1\leq s \neq t \leq p} u_{i_s} u_{i_t} \prod_{a\in \{1, 2, \ldots, p\}\setminus\{s,t\}} x_{i_a}\right)^2\nonumber\\
&\lesssim_p \frac{1}{N^{p-1}} \sum_{1\leq i_1<\ldots<i_p \leq N} \sum_{1\leq s \neq t \leq p} u_{i_s}^2 u_{i_t}^2\label{st1}\\
&\lesssim_p \frac{1}{N^{p-1}} \sum_{1\leq s\neq t \leq p} \sum_{1 \leq i_1,\ldots,i_p \leq N} u_{i_s}^2 u_{i_t}^2\nonumber\\
&= \frac{1}{N^{p-1}} \sum_{1\leq s\neq t\leq p} N^{p-2}\|\bm u\|^4 \lesssim_p \frac{1}{N}. \nonumber
\end{align}
where in \eqref{st1} we used the inequality $\left(\sum_{i=1}^n a_i\right)^2 \leq n \sum_{i=1}^n a_i^2$, for any sequence of real numbers $a_1,\ldots,a_n$.  \\ 

\noindent\textit{Proof of} $(2)$: Fix $\bm x \in \cC_N$ and $\bm u, \bm v \in S^{N-1}$.	Then, 	 		
\begin{align*}
& \e\left[G_{\bm u}(\bt) - G_{\bm v}(\bt)\right]^2\\ 
&= ((p-2)!)^2\e\left[\sum_{1\leq i_1<\ldots<i_p \leq N} J_{i_1\ldots i_p} \left(\sum_{1\leq s \neq t \leq p} (u_{i_s} u_{i_t} - v_{i_s} v_{i_t}) \prod_{a\in \{1, 2, \ldots, p\}\setminus\{s,t\}} \tau_{i_a}\right)\right]^2\\
& 
\lesssim_p \frac{1}{N^{p-1}}\sum_{1\leq i_1<\ldots<i_p \leq N} \left(\sum_{1\leq s \neq t \leq p} (u_{i_s} u_{i_t} - v_{i_s} v_{i_t}) \prod_{a\in \{1, 2, \ldots, p\}\setminus\{s,t\}} \tau_{i_a}\right)^2\\
& \lesssim_p \frac{1}{N^{p-1}}\sum_{1\leq i_1<\ldots<i_p \leq N} \sum_{1\leq s \neq t \leq p} (u_{i_s} u_{i_t} - v_{i_s} v_{i_t})^2\\ 
& \lesssim_p \frac{1}{N} \sum_{1 \leq i,j \leq N} (u_iu_j - v_iv_j)^2 \\ 
& = \frac{2}{N}  \left[1-\left(\sum_{i=1}^N u_i v_i\right)^2\right] \leq \frac{2}{N} \left[2-2\sum_{i=1}^N u_iv_i \right] \lesssim \frac{1}{N} \sum_{i=1}^N (u_i - v_i)^2.    
\end{align*} 
This completes the proof of Lemma \ref{tec2} (2). \qed \\

Using the lemma above we first show that $\e[\sup_{\bm u \in S^{N-1}} G_{\bm u}(\bt)] \lesssim_p 1$. We do this comparing the supremum of the Gaussian process $\{G_{\bm u}(\bt) : \bm u \in S^{N-1}\}$ with the supremum of the Gaussian process $\{H_{\bm u}: \bm u \in S^{N-1}\}$, where $H_{\bm u} = \sum_{i=1}^N g_i u_i$ and $g_1,\ldots,g_N$ are independent standard Gaussians. Now, by Lemma \ref{tec2} (2), there exists a constant $C:=C(p) > 0$, such that for $\bm u, \bm v \in S^{N-1}$, 
$$\e\left[G_{\bm u}(\bt) - G_{\bm v}(\bt)\right]^2 \leq \frac{C}{N} \sum_{i=1}^N (u_i - v_i)^2 =  \frac{C}{N} \e\left[H_{\bm u} - H_{\bm v}\right]^2.$$ Hence, by the Sudakov-Fernique inequality \cite[Theorem 1.1]{error}, 
\begin{align}\label{eq:Gx_expectation}
\e \left[\sup_{\bm u \in S^{N-1}} G_{\bm u}(\bt)\right]  \leq  \left(\frac{C}{N}\right)^{\frac{1}{2}} \e\left[\sup_{\bm u \in S^{N-1}} H_{\bm u}\right] & =   \left(\frac{C}{N}\right)^{\frac{1}{2}} \e\left[\left(\sum_{i=1}^N g_i^2\right)^{\frac{1}{2}}\right] \nonumber \\
&\leq  C^{\frac{1}{2}} : = D. 
\end{align}

Now, since by Lemma \ref{tec2} $(1)$ there exists a constant $K : = K(p) > 0$ such that $\sup_{\bm u \in S^{N-1}} \e \left[G_{\bm u} (\bt)^2\right] \leq \frac{K}{N}$, by the Borell-TIS inequality \cite[Theorem 2.1.1]{Gaussian}, for any $t > 0$, 
\begin{equation*}
\p\left(\sup_{\bm u \in S^{N-1}} G_{\bm u}(\bt) - \e \left[\sup_{\bm u \in S^{N-1}} G_{\bm u}(\bt)\right] > t\right) \leq e^{-\frac{N t^2}{2K} }.
\end{equation*}
This implies, by \eqref{eq:Gx_expectation},
\begin{equation*}
\p\left(\sup_{\bm u \in S^{N-1}} G_{\bm u}(\bt) > D+t\right) \leq e^{-\frac{N t^2}{2K} }.
\end{equation*} 
Then, taking $t = \sqrt{2K}$ in the inequality above gives, 
$$\p\left(\|\bm J_N(\bt)\|> D + \sqrt{2K}\right) \leq 2e^{-N},$$
since we have $\lambda_{\max}(\bm J_N(\bt)) \stackrel{D}{=} -\lambda_{\min}(\bm J_N(\bt))$, because  $\bm J_N(\bt) \stackrel{D}{=} - \bm J_N(\bt)$. Therefore, by an union bound, 
$$\p\left(\sup_{\bt \in \sa_N}\|\bm J_N(\bt)\|> D + \sqrt{2 K}\right) \leq 2^{N+1}e^{-N} = 2(e/2)^{-N}.$$ Hence, by the Borel-Cantelli lemma, $\limsup_{N\rightarrow \infty} \sup_{\bt \in \sa_N}\|\bm J_N(\bt)\| \leq D + \sqrt{2 K}$ with probability $1$, which establishes \eqref{eq:JNcondition} and hence, condition $(1)$ of Theorem \ref{chextension}.

\section{Proof of Corollary \ref{boundeddeg} } 
\label{sec:boundeddegpf}

We begin by showing that condition $(1)$ of Corollary \ref{boundeddeg} implies $\sup_{N \geq 1} \sup_{\bm x \in \cC_N} ||\bm J_N(\bm x)|| < \infty$ and, hence, condition $(1)$ of Theorem \ref{chextension}. To this end, fix $\bm x \in \cC_N$, and take $\bm u \in S^{N-1} := \{\bm t \in \mathbb{R}^N: \|\bm t\| = 1\}$. Then, 
\begin{align}\label{sh}
|\bm u^\top \bm J_N(\bm x) \bm u| = \Big|\sum_{1\leq i_1,i_2,\ldots,i_p\leq N} J_{i_1 i_2\ldots i_p} u_{i_1} u_{i_2} x_{i_3} \ldots x_{i_p}\Big| &\leq \sum_{1\leq i_1,i_2,\ldots,i_p\leq N} |J_{i_1 i_2\ldots i_p}| |u_{i_1}| |u_{i_2}| \nonumber\\&= (p-2)!~ |\bm u|^\top \bm D_{\bm J_N} |\bm u| \lesssim_p \|\bm D_{\bm J_N}\|,
\end{align} 
since $|\bm u| := (|u_1|,\ldots,|u_N|) \in S^{N-1}$. Taking supremum over all $\bm u \in S^{N-1}$ followed by the supremum over all $\bm x \in \sa_N$ and further followed by the supremum over all $N \geq 1$ throughout \eqref{sh}, we have:
$$\sup_{N \geq 1} \sup_{\bm x \in \cC_N} ||\bm J_N(\bm x)|| \lesssim_p \sup_{N\geq 1}\|\bm D_{\bm J_N}\| < \infty.$$

Next, we verify condition (2) in Theorem \ref{chextension}. To this end, we need the following lemma: 

\begin{lem}\label{oddp} For every $p \geq 2$, under the assumptions of Corollary \ref{boundeddeg}, $ |F_N^{(3)}(0)| = O(N)$, where $F^{(3)}(0)$ denotes the third derivative of $F_N(\beta)$ at $\beta=0$. 
\end{lem}

The proof of the lemma is given below. First we show how it can be used to prove condition (2) in Theorem \ref{chextension}.  To begin with, note that condition (2) of Corollary \ref{boundeddeg} implies that 
\begin{align}\label{eq:FN_II}
F_N''(0) = \mathrm{Var}_0(H_N(\bs)) = \e_0 H_N^2(\bs) = (p!)^2\sum_{1 \leq i_1<\ldots<i_p \leq N} J_{i_1\ldots i_p}^2 = \Omega(N). 
\end{align} 
Hence, $\liminf_{N\rightarrow \infty} \frac{1}{N} F_N''(0) > 0$. Now, since by Lemma \ref{oddp}, $\limsup_{N\rightarrow \infty} \frac{1}{N} |F_N^{(3)}(0)| <\infty$, we can choose $\varepsilon > 0$ small enough, such that for all $N$ large enough, 
\begin{equation*}
\frac{\varepsilon}{6N} |F_N^{(3)}(0)| < \frac{F_N''(0)}{4N}.
\end{equation*} 
Therefore, because the fourth derivative $F_N^{(4)}(b) = \e_b H_N^4(\bs) \geq 0$ for all $b \geq 0$,  a Taylor expansion gives the following for all $\beta \in (0,\varepsilon)$:
\begin{equation*}
\frac{F_N(\beta)}{N} \geq \frac{\beta^2}{2N} F_N''(0) +  \frac{\beta^3}{6N} F_N^{(3)} (0) \geq \frac{\beta^2}{2N} F_N''(0) -  \frac{\beta^3}{6N} |F_N^{(3)} (0)| > \frac{\beta^2}{4N}F_N''(0) = \Omega(1),
\end{equation*}
where the last step uses \eqref{eq:FN_II}. This verifies  condition (2) of Theorem \ref{chextension} for all $\beta > 0$, by the monotonicity of $F_N$.

\subsubsection*{Proof of Lemma \ref{oddp}} To begin with observe that $F_N^{(3)}(0)=\e_0 H_N^3(\bs)$. Now, the proof of the lemma for odd $p$ is trivial. This is because, under $\p_0$, $\bs \stackrel{D}{=} - \bs$, and for odd $p$, $H_N(-\bs) = - H_N(\bs)$, which implies  $\e_0 H_N^3(\bs)$. Hence, we will assume that $p = 2 q$, for $q \geq 1$,  throughout the rest of the proof. Now, note that 
		\begin{equation}\label{exh3}
		 \e_0H_N^3(\bs) = \sum_{\substack{1 \leq i_1, \ldots, i_p \leq N \\ \text{distinct}}}  \sum_{\substack{1 \leq j_1, \ldots, j_p \leq N \\ \text{distinct}}} \sum_{\substack{1 \leq k_1, \ldots, k_p \leq N \\ \text{distinct}}}   J_{i_1, \ldots, i_p } J_{j_1, \ldots, j_p }  J_{k_1, \ldots, k_p }   \e_0\left( \prod_{s=1}^p X_{i_s} X_{j_s} X_{k_s} \right). 
		\end{equation}
Observe for each term in the sum above, the expectation is non-zero, if only if the multiplicity of each element in the multi-set $\{ i_1, \ldots, i_p \} \bigcup \{ j_1, \ldots, j_p\} \bigcup \{ k_1, \ldots, k_p\}$ is exactly $2$. This implies, the number of distinct elements in $\{ i_1, \ldots, i_p \} \bigcup \{ j_1, \ldots, j_p\} \bigcup \{ k_1, \ldots, k_p\}$ is  $3q$ and every pair of sets among  $\{ i_1, \ldots, i_p \}$, $\{ j_1, \ldots, j_p\}$, $\{ k_1, \ldots, k_p\}$  must have exactly $q$ elements in common. Therefore, from \eqref{exh3} and recalling the definition of the matrix $\bm D_{\bm J_N}$ from \eqref{eq:dJN} we get, 
\begin{align*}
\big|\e_0 H_N^3(\bs)\big| & \lesssim_p \sum_{\substack{1 \leq i_1, \ldots, i_q, i_{q+1}, \ldots, i_{2 q}, i_{2q+1}, \ldots, i_{3q} \leq N \\ \textrm{distinct}}} \big|J_{ i_1, \ldots, i_q, i_{q+1}, \ldots, i_{2 q}} J_{i_{q+1}, \ldots, i_{2 q} i_{2q+1}, \ldots, i_{3q} } J_{ i_{2q+1}, \ldots, i_{3q} i_1, \ldots i_q}\big| \nonumber \\ 
& \lesssim_p \sum_{\substack{1 \leq i_1,i_{q +1}, i_{2q+1} \leq N \\ \textrm{distinct}}} \bm d_{\bm J_N}(i_1, i_{q+1}) \bm d_{\bm J_N}(i_{q+1}, i_{2q+1}) \bm d_{\bm J_N}(i_1, i_{2q+1})  \nonumber \\  
& = \mathrm{Trace} (\bm D_{\bm J_N}^3)  \leq N \|\bm D_{\bm J_N}\| = O(N),
\end{align*}
where the last step uses the assumption that $\sup_{N \geq 1}\|\bm D_{\bm J_N}\| < \infty$. \qed

\section{Proof of Theorem \ref{sbmthr}}\label{proof6}

We start by proving Theorem \ref{sbmthr} (1). To this end, it suffices to verify Theorem \ref{chextension}. Note that condition $(1)$ is easily satisfied because the bounded maximum degree condition \eqref{eq:DN_bound} holds almost surely: 
$$\max_{1 \leq i \leq N} d_{\bm J_N}(i)= \frac{1}{N^{p-1}} \max_{1 \leq i \leq N} d_{\bm A_{H_N}}(i) = O(1),$$
since $d_{\bm A_{H_N}}(i) \leq N^{p-1}$, for all $1 \leq i \leq N$, in any $p$-uniform hypergraph $H_N$. Next, we verify condition (2) of Theorem \ref{chextension}. To this end, by the lower bound in \cite[Theorem 1.6]{CD16} (which is the mean-field lower bound to the Gibbs variational representation of the log-partition function), we have 
\begin{align}\label{lowbdgibbs}
\e {F}_N(\beta) & \geq \E\left[\sup_{\bm x \in [-1,1]^N}\left\{ \beta  \sum_{1 \leq i_1, \ldots, i_p \leq N} \e(J_{i_1\ldots i_p}) x_{i_1}\ldots x_{i_p} - \sum_{i=1}^N I(x_i) \right\} \right] \nonumber \\ 
& \geq \sup_{\bm x \in [-1,1]^N}\left\{ \frac{\beta}{N^{p-1}} \sum_{1 \leq i_1, \ldots, i_p \leq N} \e({a}_{i_1\ldots i_p}) x_{i_1}\ldots x_{i_p} - \sum_{i=1}^N I(x_i) \right\} . 
\end{align}
Now, take any $\bm t := (t_1,\ldots,t_K) \in [-1,1]^K$, and define $\bm x \in [-1,1]^N$ by taking $x_i := t_j$, if $i \in \cB_j$, where $\cB_1, \cB_2, \ldots, \cB_K$ are as in Definition \ref{defn:block}. Then, the term inside the supremum in the RHS of \eqref{lowbdgibbs} equals $N \phi_\beta(t_1,\ldots,t_K) + O(1)$ (recall the definition of the function $\phi_\beta(t_1,\ldots,t_K)$ from \eqref{eq:threshold_function}).  Hence, \eqref{lowbdgibbs} gives us, 
\begin{equation}\label{blocklb1}
\frac{\e F_N(\beta)}{N} \geq \sup_{(t_1, t_2, \ldots, t_K) \in [0,1]^K} \phi_\beta(t_1,\ldots,t_K) + o(1).
\end{equation}
The bound in \eqref{blocklb1} above combined with the definition of the threshold $\beta_{\mathrm{HSBM}}^*$ in \eqref{eq:beta_threshold} and now implies that for all $\beta >\beta_{\mathrm{HSBM}}^*$, $\liminf_{N\rightarrow \infty} \frac{1}{N} \e F_N(\beta) > 0$. Then by Lemma \ref{mcdiarmid} below, it follows that $\liminf_{N \rightarrow \infty} \frac{1}{N}F_{N}(\beta) > 0$ with probability $1$. This verifies condition (2) of Theorem \ref{chextension}, and shows that the MPL estimate $\hat \beta_N (\bm X)$ is $\sqrt N$-consistent for $\beta >\beta_{\mathrm{HSBM}}^*$.

\begin{lem}\label{mcdiarmid}
	Let $F_N(\beta)$ denote the log-partition function of the $p$-tensor stochastic block model as in Theorem \ref{sbmthr}. Then, for every $\beta > 0$, the sequence $F_N(\beta) - \e F_N(\beta)$ is bounded in probability.  
\end{lem}

\begin{proof}
To start with, note that $ {F}_N(\beta)$ is a function of the collection of i.i.d. random variables $\mathcal{A} := \{A_{i_1\ldots i_p}\}_{1\leq i_1<\ldots < i_p \leq N}$, and so, it is convenient to denote $ {F}_N(\beta)$ by $ {F}_{N,\beta}(\mathcal{A})$. Let us take $\mathcal{A}' := \{A'_{i_1\ldots i_p}\}_{1\leq i_1<\ldots< i_p\leq N}$, where $A'_{123\ldots p} = 1- A_{123\ldots p}$ and $A'_{i_1\ldots i_p} = A_{i_1\ldots i_p}$ for all $(i_1,\ldots,i_p) \neq (1,2,3,\ldots,p)$.  Note that 
\begin{equation*}
\left| \sum_{i_1<\cdots<i_p} A_{i_1\ldots i_p} X_{i_1}\cdots X_{i_p} - \sum_{i_1<\cdots<i_p} A_{i_1\ldots i_p}' X_{i_1}\cdots X_{i_p}\right|= |X_1 X_2\ldots X_p| = 1.
\end{equation*} 
Hence,  
$$\exp\left\{\frac{\beta p! }{N^{p-1}} \sum_{i_1<\cdots<i_p} A_{i_1\ldots i_p} X_{i_1}\cdots X_{i_p}\right\} \leq \exp\left\{\frac{\beta p! }{N^{p-1}}\right\}\exp\left\{\frac{\beta p!}{N^{p-1}} \sum_{i_1<\cdots<i_p} A_{i_1\ldots i_p}' X_{i_1}\cdots X_{i_p}\right\}.$$
The above inequality implies that $ {F}_{N,\beta}(\mathcal{A}) \leq  {F}_{N,\beta}(\mathcal{A}') + \beta p! N^{1-p}$. Similarly, we also have $ {F}_{N,\beta}(\mathcal{A}') \leq  {F}_{N,\beta}(\mathcal{A}) + \beta p!  N^{1-p}$, and hence, $$\left| {F}_{N,\beta}(\mathcal{A}) -  {F}_{N,\beta}(\mathcal{A}')\right| \leq \beta p!  N^{1-p}.$$ Of course, the above arguments hold if $\mathcal{A}'$ is obtained by flipping any arbitrary entry of $\mathcal{A}$ (not necessarily the $(1,2,\ldots,p)$-th entry) and keeping all other entries unchanged. Hence, the assumption of McDiarmid's inequality \cite{mcdiarmid} holds with bounding constants $c_{i_1\ldots i_p} = \beta p! N^{1-p}$. Therefore, for every $t > 0$:
$$\mathbb{P}\left(\left| {F}_N(\beta) - \mathbb{E}  {F}_N(\beta)\right| \geq t\right) \leq 2\exp\left\{-\frac{2t^2}{\sum_{ i_1<\ldots<i_p}c_{i_1\ldots i_p}^2}\right\} \leq 2\exp\left\{-\frac{2t^2 N^{p-2}}{\beta^2(p!)^2}\right\},$$ 
which completes the proof of the lemma. 	 
\end{proof}

We will now use Lemma \ref{mcdiarmid} to prove Theorem \ref{sbmthr} (2). To this end, we will show that
\begin{align}\label{eq:FNbeta_H}
\e F_N(\beta) = O(1), \quad \text{ for } \beta < \beta_{\mathrm{HSBM}}^*, 
\end{align} 
the expectation in \eqref{eq:FNbeta_H} being taken with respect to the randomness of the HSBM. To see why this implies Theorem \ref{sbmthr} (2), assume, on the contrary, that there is a sequence of estimates which is consistent for $\beta < \beta_{\mathrm{HSBM}}^*$.  Using this sequence of estimates we can then construct a consistent sequence of tests $\{\phi_N\}_{N \geq 1}$ for the following hypothesis testing problem:\footnote{A sequence of tests $\{\phi_N\}_{N \geq 1}$ is said to be {\it consistent} 
	if both its Type I and Type II errors converge to zero as $N \rightarrow \infty$, that is, $\lim_{N\rightarrow\infty}\E_{H_0}\phi_N=0$, and the power $\lim_{N\rightarrow\infty}\E_{H_1}\phi_N=1$.}  
\begin{align}\label{eq:hypothesis} 
H_0: \beta = \beta_1\quad\quad\textrm{versus}\quad\quad H_1: \beta=\beta_2, 
\end{align}
if $\beta_1 < \beta_2 < \beta_{\mathrm{HSBM}}^*$. To this end, denote by $\mathbb{Q}_{\beta,p}$ the joint distribution of the HSBM and the $p$-tensor Ising model with parameter $\beta$.  Then a simple calculation shows that for any two positive real numbers $\beta_1 < \beta_2$, the Kullback-Leibler (KL) divergence between the joint measures $\mathbb{Q}_{\beta_1,p}$ and $\mathbb{Q}_{\beta_2,p}$ is given by:
\begin{align}\label{eq:DKL}
D_N(\mathbb{Q}_{\beta_1,p}\| \mathbb{Q}_{\beta_2,p}) = \e D_N(\p_{\beta_1,p}\| \p_{\beta_2,p}) =  \e F_N(\beta_2) - \e F_N(\beta_1) -(\beta_2-\beta_1) \e F_N'(\beta_1),
\end{align} 
where, as before, the expectation in \eqref{eq:DKL} is taken with respect to the randomness of the HSBM.  
Now, by the monotonicity of $F_N'(\cdot)$,		
$$0 = (\beta_2-\beta_1)F_N'(0) \leq (\beta_2-\beta_1)F_N'(\beta_1)\leq \int_{\beta_1}^{\beta_2} F_N'(t)\mathrm dt = F_N(\beta_2)-F_N(\beta_1).$$ 
Hence, by  \eqref{eq:FNbeta_H} and \eqref{eq:DKL}, $D_N(\mathbb{Q}_{\beta_1,p}\| \mathbb{Q}_{\beta_2,p}) = O(1)$. Then, by \cite[Proposition 6.1]{BM16}, there cannot exist any sequence of consistent tests for the hypothesis \eqref{eq:hypothesis}, which leads to a contradiction. This completes the proof of Theorem \ref{sbmthr} (2).

\subsection{Proof of (\ref{eq:FNbeta_H})} The proof of \eqref{eq:FNbeta_H} has the following two steps: 
\begin{itemize}
	\item[(I)] Define a new $p$-tensor Ising model on $N$ nodes, with interaction tensor $\tilde{\bm J}_N := \e {\bm J_N}$. We will call this model $\mathcal M_0 $. The first step in the proof of \eqref{eq:FNbeta_H} is to show that the log-partition function $\tilde{F}_N$ of the model $\mathcal M_0 $ is bounded, for every $\beta < \beta_{\mathrm{HSBM}}^*$.
	
	\item[(II)] The second step is to show that the expected log-partition function $\e F_N(\beta)$ of the original model is bounded, for $\beta < \beta_{\mathrm{HSBM}}^*$, by comparing it with the log-partition function $\tilde{F}_N$ of the model $\mathcal M_0 $. The result in \eqref{eq:FNbeta_H} then follows by an application of Lemma \ref{mcdiarmid}. 
\end{itemize}

\subsubsection{Proof of Step \em{(I)}}

Throughout this section we fix $\beta < \beta_{\mathrm{HSBM}}^*$ and denote by $\p_{\beta,\mathcal M_0 }$ the probability measure corresponding to the model $\mathcal M_0 $ at the parameter $\beta$ and $\e_{\beta, \mathcal M_0 }$  the expectation with respect to the probability measure $\p_{\beta,\mathcal M_0 }$.

\begin{lem}\label{lm:logpartition_I} Denote by $\tilde{F}_N(\beta)$ the log-partition function of the model $\cM_0$. Then for $\beta < \beta_{\mathrm{HSBM}}^*$, $\limsup_{N \rightarrow \infty} \tilde{F}_N(\beta) < \infty $. 
\end{lem}

\noindent {\it Proof of Lemma} \ref{lm:logpartition_I}:  Denote the Hamiltonian of the model $\mathcal M_0 $ by $\tilde{H}_N(\bs)$, that is, 
$$\tilde{H}_N(\bs) : = \frac{1}{N^{p-1}} \sum_{1 \leq i_1, i_2, \ldots, i_p \leq N} \E(a_{i_1 i_2 \ldots i_p}) X_{i_1} X_{i_2} \ldots X_{i_p}.$$  
For each $\bs \in \sa_N$ and $1 \leq j \leq K$, define $S_j(\bs):= \sum_{i \in \cB_j} X_i$. With these notations, we have $\tilde{H}_N(\bs) = \overline{H}_N(\bs) + O(1)$, where
\begin{align}\label{eq:HN_tensor}
\overline{H}_N(\bs) = \frac{1}{N^{p-1} }  \sum_{1 \leq j_1,\ldots, j_p \leq K} \theta_{j_1\ldots j_p}  \prod_{\ell=1}^{p} S_{j_\ell}(\bs).
\end{align}  
Let us define $\overline{Z}_N(\beta) := \frac{1}{2^N} \sum_{\bs \in \sa_N} e^{\beta \overline{H}_N(\bs)}$. Since $\tilde{F}_N(\beta) = \log \overline{Z}_N(\beta) + O(1)$, it suffices to show that $\overline{Z}_N(\beta) = O(1)$. Towards this, for each $1 \leq j\leq K$, define the sets:
$$I_j:= \left\{-1,~ -1+\frac{2}{|\cB_j|},~ -1+ \frac{4}{|\cB_j|},~\ldots,~ 1-\frac{2}{|\cB_j|},~ 1\right\}\quad~\textrm{and}\quad A_s(j) := \{\bs \in \sa_N: S_j(\bs) = s\}.$$ 
Recall, $\cB_j=(N \sum_{i=1}^{j-1}\lambda_i, N \sum_{i=1}^j \lambda_i] \bigcap [N]$, hence $| |\cB_j| - N \lambda_j |\leq 2$. Now, note that 
\begin{align}
\overline{Z}_N(\beta) & = \frac{1}{2^N} \sum_{(\ell_1,\ldots,\ell_K) \in I_1 \times \cdots \times I_K} e^{\frac{\beta}{N^{p-1}} \sum_{1\leq j_1,\ldots,j_p\leq K} \theta_{j_1 \ldots j_p} \prod_{m=1}^p \ell_{j_m} |\cB_{j_m}|} \left| \bigcap_{j=1}^K A_{\ell_j |\cB_j|}(j)\right| \nonumber \\ 
\label{fe} &\lesssim \frac{1}{2^N} \sum_{(\ell_1,\ldots,\ell_K) \in I_1 \times \cdots \times I_K} e^{N\beta \sum_{1\leq j_1,\ldots,j_p\leq K} \theta_{j_1 \ldots j_p} \prod_{m=1}^p \lambda_{j_m} \ell_{j_m}} \left| \bigcap_{j=1}^K A_{\ell_j |\cB_j|}(j)\right|\\ 
\label{eq:T12} &= T_1 + T_2 ,
\end{align}
where the term $T_1$ is obtained by restricting the sum in the RHS of \eqref{fe} to the set $(I_1\times\cdots \times I_K) \bigcap [-\frac{1}{2} , \frac{1}{2}]^K$ and the term $T_2$ is the sum restricted to the set $(I_1\times\cdots \times I_K) \bigcap ([-\frac{1}{2} , \frac{1}{2}]^K)^c$. 

Let us bound $T_1$ first. Note that  
$$|A_{\ell_j |\cB_j|}(j)|= {|\cB_j| \choose \frac{|\cB_j|(1+\ell_j)}{2}}.$$
Then by the Stirling's approximation of the binomial coefficient (see, for example, \cite[Lemma B.5]{mlepaper}) and using the fact that the sets $\cB_1, \cB_2, \ldots, \cB_K$ are disjoint, we have for all $(\ell_1,\ldots,\ell_K) \in [-\frac{1}{2} , \frac{1}{2}]^K$,
$$\left| \bigcap_{j=1}^K A_{\ell_j |\cB_j|}(j)\right| = 2^N \exp\left\{-\sum_{j=1}^K |\cB_j| I(\ell_j) \right\} O\left(\frac{1}{\sqrt{\prod_{j=1}^K |\cB_j|}}\right).$$
Hence, denoting $\bm \ell := (\ell_1,\ldots,\ell_K)$ and $\cI := (I_1 \times \cdots \times I_K)$ gives, 
\begin{align*}
T_1 &=  \sum_{\bm \ell \in \cI \bigcap \left[-\frac{1}{2} , \frac{1}{2}\right]^K} e^{ N\beta \sum_{1\leq j_1,\ldots,j_p\leq K} \theta_{j_1 \ldots j_p} \prod_{m=1}^p \lambda_{j_m} \ell_{j_m} -\sum_{j=1}^K |\cB_j| I(\ell_j) } O\left(\frac{1}{\sqrt{\prod_{j=1}^K |\cB_j|}}\right)\\ 
&\lesssim \sum_{\bm \ell \in \cI \bigcap \left[-\frac{1}{2} , \frac{1}{2}\right]^K} e^{ N\beta \sum_{1\leq j_1,\ldots,j_p\leq K} \theta_{j_1 \ldots j_p} \prod_{m=1}^p \lambda_{j_m} \ell_{j_m} - N\sum_{j=1}^K \lambda_j I(\ell_j) } O\left(\frac{1}{\sqrt{\prod_{j=1}^K |\cB_j|}}\right)\\&= \sum_{\bm \ell \in \cI\bigcap \left[-\frac{1}{2} , \frac{1}{2}\right]^K} e^{N\phi_\beta(\ell_1,\ldots,\ell_K) } O\left(\frac{1}{\sqrt{\prod_{j=1}^K |\cB_j|}}\right)
\end{align*} 
Now, since $\beta < \beta_{\mathrm{HSBM}}^*$ we can choose $\varepsilon \in (0,1)$, such that $\beta/(1 - \varepsilon) < \beta_{\mathrm{HSBM}}^*$. Hence, recalling \eqref{eq:beta_threshold}, 
$$\phi_{\beta/(1-\varepsilon)}(\ell_1,\ldots,\ell_K) \leq 0,\quad\textrm{that is, }\quad
\phi_{\beta}(\ell_1,\ldots,\ell_K) \leq -\varepsilon \sum_{j=1}^K \lambda_j I(\ell_j) \leq -\frac{\varepsilon}{2} \sum_{j=1}^K \lambda_j \ell_j^2,$$
where the last step uses $I(x) \geq x^2/2$.  Then by a Riemann sum approximation (see \cite[Lemma B.2]{mlepaper}), 
\begin{align} 
\label{eq:T11}T_1 &\lesssim \sum_{\bm \ell \in \cI\bigcap \left[-\frac{1}{2} , \frac{1}{2}\right]^K} e^{-\frac{N\varepsilon}{2} \sum_{j=1}^K \lambda_j \ell_j^2 } O\left(\frac{1}{\sqrt{\prod_{j=1}^K |\cB_j|}}\right)  \\ 
&\leq  \frac{1}{2^{K}} O\left(\sqrt{\prod_{j=1}^K |\cB_j|}\right) \int\limits_{\mathbb{R}^K} \exp\left\{-\frac{N\varepsilon}{2} \sum_{j=1}^K \lambda_j x_j^2 \right\}~\mathrm dx_1\ldots \mathrm dx_K + O(1) \nonumber \\ 
&= \frac{1}{2^{K}} O\left(\sqrt{\prod_{j=1}^K \frac{|\cB_j|}{N}}\right) \int\limits_{\mathbb{R}^K} \exp\left\{-\frac{\varepsilon}{2} \sum_{j=1}^K \lambda_j y_j^2 \right\}~\mathrm dy_1\ldots \mathrm dy_K + O(1) \nonumber \\ 
\label{eq:T1} &\lesssim \int\limits_{\mathbb{R}^K} \exp\left\{-\frac{\varepsilon}{2} \sum_{j=1}^K \lambda_j y_j^2 \right\}~\mathrm dy_1\ldots \mathrm dy_K + O(1) = O(1) .
\end{align} 

Now, we bound $T_2$. For this we need the following combinatorial estimate:
\begin{obs}\cite[Equation (5.4)]{talagrand} For every integer $s$ and positive integer $m$, 
	\begin{equation}\label{new}
	\left|\left\{\bm x \in \{-1, 1\}^m: \sum_{i=1}^m x_i =  s\right\}\right| \leq 2^{m} e^{-m I\left(\frac{s}{m}\right)}. 
	\end{equation}  
\end{obs}	
	
Using the bound in \eqref{new} and recalling that the sets $\cB_1, \cB_2, \ldots, \cB_K$ are disjoint, we have for every $(\ell_1, \ell_2, \ldots, \ell_K) \in (I_1, I_2, \ldots, I_K)$, 
 \begin{equation*}
\left| \bigcap_{j=1}^K A_{\ell_j |\cB_j|}(j)\right|  \leq 2^N   \exp\left\{-\sum_{j=1}^K |\cB_j| I(\ell_j)\right\}.
\end{equation*}
Hence, by following the arguments used to obtain \eqref{eq:T11} we get,  
\begin{align}\label{eq:T2}
T_2 \lesssim \sum_{\bm \ell \in \cI\bigcap \left(\left[-\frac{1}{2} , \frac{1}{2}\right]^K\right)^c} e^{-\frac{N\varepsilon}{2} \sum_{j=1}^K \lambda_j \ell_j^2 } &\leq \left(\prod_{j=1}^K (|\cB_j|+1)\right)\max_{\bm \ell \in \cI\bigcap \left(\left[-\frac{1}{2} , \frac{1}{2}\right]^K\right)^c}  e^{-\frac{N\varepsilon}{2} \sum_{j=1}^K \lambda_j \ell_j^2 } 
\nonumber\\
&\leq \left(\prod_{j=1}^K(N\lambda_j+1)\right)\max_{1\leq j\leq K}  e^{-\frac{N\varepsilon \lambda_j}{8}} = o(1). 
\end{align}
This shows, combining \eqref{eq:T12}, \eqref{eq:T1}, and \eqref{eq:T2},  that $\overline{Z}_N(\beta) = O(1)$ for all $\beta < \beta_{\mathrm{HSBM}}^*$, completing the proof of Lemma \ref{lm:logpartition_I}.

\subsubsection{Proof of Step \em{(II)}}  We now show that $\e F_N(\beta)$ is bounded, for $\beta < \beta_{\mathrm{HSBM}}^*$, which will allow us to conclude, using Lemma \ref{mcdiarmid}, that  $F_N(\beta) = O(1)$ with probability $1$, for all $\beta < \beta_{\mathrm{HSBM}}^*$, that is, \eqref{eq:FNbeta_H} holds. 

\begin{lem}\label{st2} For every $\beta < \beta_{\mathrm{HSBM}}^*$, $\limsup_{N\rightarrow \infty} \e F_N(\beta) < \infty$. 
\end{lem}

\begin{proof} Fix $\beta < \beta_{\mathrm{HSBM}}^*$. Then the partition function $Z_N(\beta)$ becomes: 
	\begin{align*}
	Z_N(\beta) &= \frac{1}{2^N} \sum_{\bs \in \sa_N} e^{\frac{p! \beta}{N^{p-1}} \sum_{i_1<\cdots<i_p} A_{i_1\ldots i_p} X_{i_1}\cdots X_{i_p}}\\
	&= \frac{1}{2^N} \sum_{\bs \in \sa_N} e^{\frac{p! \beta}{N^{p-1}}  \sum_{i_1<\cdots<i_p} \left(A_{i_1\ldots i_p} - \e A_{i_1\ldots i_p}\right) X_{i_1}\cdots X_{i_p}}e^{\frac{p! \beta}{N^{p-1}}  \sum_{i_1<\cdots<i_p} \e A_{i_1\ldots i_p}  X_{i_1}\cdots X_{i_p}}.
	\end{align*}
	By Hoeffding's inequality, for each $\bs \in \sa_N$ and $i_1<\ldots<i_p$, 
	\begin{align*}
	\mathbb{E} \exp\left\{\frac{p! \beta}{N^{p-1}} \sum_{i_1<\cdots<i_p} \left(A_{i_1\ldots i_p} - \e A_{i_1\ldots i_p}\right) X_{i_1}\cdots X_{i_p}\right\}  
	&\leq \prod_{i_1<\cdots<i_p}\exp\left\{\frac{\beta^2(p!)^2}{8 N^{2p-2}}\right\}\\
	&\leq \exp\left\{\frac{\beta^2 p!}{8 N^{p-2}}\right\}.
	\end{align*}
	This shows that 
	$$\mathbb{E} Z_N(\beta) \leq \exp\left\{\frac{\beta^2 p!}{8 N^{p-2}}\right\} \left[\frac{1}{2^N} \sum_{\bs \in \sa_N}e^{\frac{p! \beta}{N^{p-1}}\sum_{i_1<\cdots<i_p} \e A_{i_1\ldots i_p} X_{i_1}\cdots X_{i_p}}\right] = \exp\left\{\frac{\beta^2 p!}{8 N^{p-2}}\right\}\tilde{Z}_N(\beta),$$
	where $\tilde{Z}_N$ is the partition function of the model $\mathcal M_0 $. Now, taking logarithms and using Lemma \ref{lm:logpartition_I} shows, $\limsup_{N \rightarrow \infty}\log \mathbb{E} Z_N(\beta) < \infty$, for $\beta < \beta_{\mathrm{HSBM}}^*$. Then, by Jensen's inequality, we conclude that $\limsup_{N \rightarrow \infty}\E F_N(\beta) < \infty$, for $\beta < \beta_{\mathrm{HSBM}}^*$, completing the proof of the lemma.  
\end{proof}

\section{Proofs from Section \ref{cltsec}}\label{sec:pfcwclt}

\subsection{Proof of Theorem \ref{thm:cwmplclt}}\label{sec:pf_cwmplclt}

Define the function $\phi_p: [-1,1]\mapsto (-\infty,\infty]$ as:  
\begin{align}\label{eq:2} 
\phi(t) = \phi_p(t) := 
\begin{cases}
p^{-1} t^{1-p}\tanh^{-1} (t) &\quad\text{if}~p~\textrm{is even and}~t \neq 0,\\
p^{-1} t^{1-p}\tanh^{-1} (t) &\quad\text{if}~p~\textrm{is odd and}~t > 0.\\
\infty &\quad\text{if}~p~\textrm{is odd and}~t < 0.\\
0 &\quad\text{if}~t = 0.\\
\end{cases}
\end{align}
Note that for every $t \neq 0$ when $p$ is even, and every $t > 0$ when $p$ is odd, the function $\vp$ is twice differentiable on some compact set containing $t$ in its interior, and $$\vp'(t) = -\frac{1}{pt^{p-1}}\left\{\frac{(p-1)\tanh^{-1}(t)}{t} -\frac{1}{1-t^2}\right\}.$$
Hence,  recalling the definition of the function $g$ from the statement of the theorem we have, 
\begin{equation}\label{derexp}
\vp'(m_*) = -\frac{g''(m_*)}{pm_*^{p-1}}.
\end{equation}
Moreover, from the definition of the MPLE in \eqref{mple} it is easy to see that in the Curie-Weiss model $\hat{\beta}_N(\bm X) = \phi(\cs)$. Note that $\vp(m_*) = \beta$, since $g'(m_*) = 0$. This implies, 
\begin{equation}\label{delstep}
\sqrt{N}(\hat{\beta}_N(\bm X) - \beta) = \sqrt{N}\left(\vp(\cs) - \vp(m_*)\right).
\end{equation}
We now consider the case $p$ is even and $p$ is odd separately. Throughout, we assume $\beta > \beta_{\mathrm{CW}}^*$: 
\begin{itemize}

\item {\it $p \geq 3$ is odd}: In this case, $m_*$ is the unique global maximizer of the function  $g$ on $[-1, 1]$. This implies, by \cite[Theorem 2.1~(1)]{mlepaper}, 
$$\sqrt{N}(\bar X_N - m_*) \dto N\left(0, -\frac{1}{g''(m_*)}\right). $$
Hence, by \eqref{delstep} and the delta method \cite[Theorem 1.8.12]{tpe}, 
$$\sqrt{N}(\hat{\beta}_N(\bm X) -\beta) ) \xrightarrow{D} N\left(0, -\frac{(\phi'(m_*))^2}{g''(m_*)}\right) \stackrel{D}= N\left(0, -\frac{g''(m_*)}{p^2m_*^{2p-2} }\right),$$
where the last step uses \eqref{derexp}. This completes the proof when $p$ is odd. 

\item {\it $p \geq 2$ is even}: In this case, the function  $g$ has two (non-zero) global maximizers on $[-1, 1]$, which are given by $m_*$ and $-m_*$ (as shown in \cite[Section C]{mlepaper}). This implies, by \cite[Theorem 2.1~(2)]{mlepaper} and the delta method, 
\begin{align}\label{eq:curie_weiss_mple_I}
\sqrt{N}(\hat{\beta}_N(\bm X) -\beta ) \Big| \{\cs \in (0, 1] \} \xrightarrow{D} N\left(0, -\frac{g''(m_*)}{p^2m_*^{2p-2} }\right). 
\end{align}
Similarly, observing that $\phi(m_*)=-\phi(-m_*)$ and $g''(m_*)=g''(-m_*)$, 
\begin{align}\label{eq:curie_weiss_mple_II}
\sqrt{N}(\hat{\beta}_N(\bm X) -\beta ) \Big| \{ \cs \in [-1, 0] \} \xrightarrow{D} N\left(0, -\frac{g''(m_*)}{p^2m_*^{2p-2} }\right). 
\end{align}
Combining \eqref{eq:curie_weiss_mple_I} and \eqref{eq:curie_weiss_mple_II}, gives the desired result when $p$ is even.

\end{itemize}

\subsection{Proof of Theorem \ref{thm:cw_threshold}}\label{sec:pf_cw_threshold} Fix $\beta = \beta_{\mathrm{CW}}^*(p)$.  We now consider the three cases in Theorem \ref{thm:cw_threshold} separately. 

\noindent{\it Proof of }(1): In this case, $p=2$ and, hence, by \eqref{eq:2}, $\hat{\beta}_N = \frac{\tanh^{-1} (\cs)}{2\cs}  \bm 1 \{\cs \neq 0\}$. Therefore, on the event $\cE_N:=\{\cs\neq 0, |\cs|< \frac{1}{2}\}$, 
	\begin{align}\label{inf1}
	N^\frac{1}{2}\left(\hat{\beta}_N -\frac{1}{2}\right) &= N^\frac{1}{2} \left(\frac{\tanh^{-1} (\cs)}{2\cs} - \frac{1}{2}\right)\nonumber\\ &= \tfrac{1}{2} N^\frac{1}{2}\sum_{s=1}^\infty \frac{\cs^{2s}}{2s+1} = \tfrac{1}{6} N^\frac{1}{2} \cs^2 + N^\frac{1}{2} O(\cs^4). 
	\end{align}
From \cite[Proposition 4.1]{comets} we know that $N^\frac{1}{4} \cs \xrightarrow{D} F$, where $F$ is as defined in the statement of Theorem \ref{thm:cw_threshold}. The result in \eqref{eq:threshold_F} now follows from \eqref{inf1}, and the observation that $\p_\beta(\cE_N^c) = o(1)$, since $\cs \pto 0$ and $\p_\beta(\cs =0 )=o(1)$ (from the proof of \cite[Lemma C.6]{mlepaper}). \\ 
	
\noindent{\it Proof of }(2): Assume $p \geq 3$. To begin with, define the the three intervals $A_1 := [-1,~-\frac{m_*}{2}]$ and $A_2 := (-\frac{m_*}{2},~\frac{m_*}{2})$ and $A_3 := [\frac{m_*}{2},~1]$ (recall that $m_*=m_*(p, \beta)$ is the unique positive maximizer of the function $g(t) := \beta t^p - I(t)$. We now consider the following two cases depending on whether $p \geq 3$ is odd or even: 
	
\begin{itemize}
\item[$\bullet$] {\it $p \geq 4$ is even:}  In this case, the function  $g$ has three global maximizers on $[-1, 1]$, which are given by $m_* > 0$ and $-m_*$, and $0$. Then, by \cite[Theorem 2.1 (2)]{mlepaper} and arguments as in \eqref{eq:curie_weiss_mple_I}, we have 
\begin{equation}\label{evp}
\sqrt{N}(\hat{\beta}_N(\bm X) - \beta)\big|\{\cs \in A_i\} \xrightarrow{D} N\left(0, -\frac{g''(m_*)}{p^2m_*^{2p-2}}\right), 
\end{equation}
for $i \in \{1, 3\}$. Next, recalling $\beta_N(\bm X) =\phi(\bar X_N)$, where $\phi(\cdot)$ as in \eqref{eq:2}, note that 
\begin{equation}\label{evp2}
\hat{\beta}_N(\bm X)\big|\{\cs \in A_2 \backslash \{0\}\} \xrightarrow{D} \infty,
\end{equation}
since $\frac{\tanh^{-1}(t)}{t^{p-1}} \rightarrow \infty$, as $|t| \rightarrow 0$. Now, since $\p (\cs=0) = o(1)$ and 
$\p_\beta(\cs \in A_2) \rightarrow \alpha$ (by \cite[Theorem 2.1 (2)]{mlepaper}), combining \eqref{evp} and \eqref{evp2} the result in \eqref{eq:threshold_mple} follows, if $p \geq 4$ is even.  

\item[$\bullet$] {\it $p \geq 3$ is odd:}   In this case, the function  $g$ has two global maximizers on $[-1, 1]$ one of which is non-positive and the other is $m_*> 0$. Then, by similar arguments as above, 
\begin{equation}\label{odp}
\sqrt{N}(\hat{\beta}_N(\bm X) - \beta)\big|\{\cs \in A_3\} \xrightarrow{D} N\left(0, -\frac{g''(m_*)}{p^2m_*^{2p-2}}\right). 
\end{equation}
Moreover, recalling \eqref{eq:2}, 
\begin{equation}\label{odp1}
\hat{\beta}_N(\bm X)\big|\{\cs \in \left(A_1\bigcup A_2 \right) \backslash \{0\}\} \xrightarrow{D} \infty.
\end{equation}
The result in \eqref{eq:threshold_mple} now follows from \eqref{odp}, \eqref{odp1}, and the fact that $\p_\beta(\cs \in A_1\bigcup A_2) \rightarrow \alpha$, if $p$ is odd (by \cite[Theorem 2.1 (2)]{mlepaper}).  \\ 
\end{itemize} 

\noindent{\it Proof of }(3): Here, we prove the finer asymptotics of $\hat{\beta}_N(\bm X)$. For this, note 
by \cite[Theorem 2.1 (2)]{mlepaper} that $\sqrt{N}\cs \big|\{\cs \in \cB\} \xrightarrow{D} N(0,1)$ for any interval $\cB$ containing $0$, but no other maximizer of $g$. We now consider the following two cases: 
\begin{itemize}
	\item [--] {\it $p\geq 4$ is even:} In this case, $\sqrt{N}\cs \big|\{\cs \in A_2\} \xrightarrow{D} N(0,1)$. Then, \eqref{eq:2} gives, 
\begin{equation*}
N^{1-\frac{p}{2}}\hat{\beta}_N(\bm X) = \frac{1}{p(\sqrt{N}\cs)^{p-2}}\cdot \frac{\tanh^{-1}(\cs)}{\cs} \bm 1\{\cs \neq 0\}.
\end{equation*}
Now, since $\frac{\tanh^{-1}(\cs)}{\cs} \bm 1 \{\cs \neq 0\} \big|\{\cs \in A_2\setminus\{0\} \} \xrightarrow{P} 1$, we have,   
	\begin{equation}\label{eq:threshold_II}
	N^{1-\frac{p}{2}} \hat{\beta}_N(\bm X) \big|\{\cs \in A_2\setminus \{0\} \} \xrightarrow{D} \frac{1}{p Z^{p-2}}, 
	\end{equation} 
	where $Z \sim N(0, 1)$. The result in now \eqref{eq:threshold_mple_I} follows from \eqref{eq:threshold_II}, by noting $\p (\cs=0) = o(1)$, that $\p_\beta(\cs \in A_2) \rightarrow \alpha$ and $N^{1-\frac{p}{2}} \hat{\beta}_N(\bm X) \big|\{\cs \in A_1\bigcup A_3\} \xrightarrow{P} 0$ (by  \eqref{evp}).

	\item[--] {\it $p\geq 3$ is odd:} In this case, since the function  $g$ has two global maximizers on $[-1, 1]$ one of which is non-positive and the other is $m_*> 0$, $\sqrt{N}\cs |\{\cs \in A_1\bigcup A_2\} \xrightarrow{D} N(0,1)$. Then, \eqref{eq:2} gives, 
\begin{equation}\label{eq:beta_p_mple}
N^{1-\frac{p}{2}}\hat{\beta}_N(\bm X) = \frac{1}{p(\sqrt{N}\cs)^{p-2}}\cdot \frac{\tanh^{-1}(\cs)}{\cs} \bm 1\{\cs > 0\} + \infty \bm 1\{\cs < 0\},
\end{equation}
where we adopt the convention infinity times zero is zero. It follows from \eqref{eq:beta_p_mple}, that $N^{1-p/2}\hat{\beta}_N(\bm X)$ is a non-negative random variable. Hence, denoting $\cA_{12}:=\{\cs \in A_1\bigcup A_2\}$ and taking $t \in [0,\infty)$ gives 
\begin{align*}
\p\left(N^{1-\frac{p}{2}} \hat{\beta}_N(\bm X)\leq t \big|\cA_{12} \right) 
&= \p\left(N^{1-\frac{p}{2}} \hat{\beta}_N(\bm X)\leq t,~\cs > 0 \big|\cA_{12}  \right) + o(1), 
\end{align*}
since $\p_\beta(\cs = 0| \cA_{12}) = o(1)$. Now, note that 
\begin{align*}
& \p\left( N^{1-\frac{p}{2}} \hat{\beta}_N(\bm X) \leq t,~\cs > 0 \big|\cA_{12}  \right) \nonumber \\ 
&= \p\left( \frac{1}{p(\sqrt{N}\cs)^{p-2}}\cdot \frac{\tanh^{-1}(\cs)}{\cs} \leq t,~\cs \neq 0 \big| \cA_{12}  \right) - \p_\beta(\cs < 0 \big| \cA_{12}) ,
\end{align*}
since on the the event $\{\cs < 0\}$, $\frac{1}{p(\sqrt{N}\cs)^{p-2}}\cdot \frac{\tanh^{-1}(\cs)}{\cs} < 0 \leq t$. Note that $\p_\beta(\cs < 0 \big| \cA_{12}) \rightarrow \frac{1}{2}$, since $\sqrt{N}\cs | \cA_{12}  \xrightarrow{D} N(0,1)$. Then, by arguments as in \eqref{eq:threshold_II}, it follows that 
\begin{align*}
\p\left( N^{1-\frac{p}{2}} \hat{\beta}_N(\bm X) \leq t \big|\cA_{12}  \right)  \rightarrow  \p\left(\frac{1}{pZ^{p-2}} \leq t\right) -\frac{1}{2} 
& = \frac{1}{2}\p\left(\frac{1}{p|Z|^{p-2}} \leq t\right).
\end{align*}
Therefore, by \eqref{eq:beta_p_mple},
	\begin{equation}\label{eq:threshold_p}
	N^{1-\frac{p}{2}} \hat{\beta}_N(\bm X) \big|\cA_{12} \xrightarrow{D} \frac{1}{2} \left(\frac{1}{p |Z|^{p-2}}\right) +\frac{1}{2}\delta_\infty.
	\end{equation}
The result in \eqref{eq:threshold_mple_II} follows now from \eqref{eq:threshold_p}, by noting that $\p_\beta(\cs \in A_1\bigcup A_2) \rightarrow \alpha$ and $N^{1-\frac{p}{2}} \hat{\beta}_N(\bm X) \big|\{\cs \in A_3\} \xrightarrow{P} 0$ (by \eqref{odp}). 
\end{itemize}

\appendix

\section{Properties of the Curie-Weiss Threshold} 
\label{cwtp}

Here, we will prove various properties of the Curie-Weiss threshold $\bw=\beta_{\mathrm{ER}}^*(p, 1)$ (recall \eqref{hypergraph_random}). 

\begin{lem} The Curie-Weiss threshold $\bw$ has the following properties: 
	\begin{itemize}
		\item[(1)] $\lim_{p \rightarrow \infty} \bw = \log 2$. 
		
		\item[(2)] The sequence $\{\bw\}_{p\geq 2}$ is strictly increasing. 
		
		\item[(3)]  $\beta_{\mathrm{CW}}^*(2) = 0.5$. 
	\end{itemize} 
\end{lem}

\begin{proof} Define the function $g_{\beta,p}(t) := \beta t^p - I(t).$ Since $g_{\beta,p}$(1)$ = \beta - \log 2$, recalling \eqref{hypergraph_random},  it immediately follows that $\bw \leq \log 2$. Now, take any $\beta < \log 2$. Note that $g_{\beta,2}$(1)$ < 0$ and the function $t \mapsto g_{\beta,2}(t)$ is continuous at $1$. Therefore, there exists $r \in (0,1)$, such that $g_{\beta,2}(t) < 0$ for all $t \in [r,1]$. Clearly, $g_{\beta,p}(t) \leq g_{\beta,2}(t) <0$ for all $p\geq 2$ and $t \in [r,1]$. Now, note that for all $t \in [0,1)$,
	\begin{equation}\label{conx}
	g_{\beta,p}''(t) = \beta p (p-1) t^{p-2} - \frac{1}{1-t^2} \leq  \beta p (p-1) t^{p-2} - 1.
	\end{equation}
	Since $\lim_{p\rightarrow \infty} p(p-1) r^{p-2} = 0$, there exists $p(\beta) \geq 2$, such that $g_{\beta,p}''(r) < 0$ for all $p \geq p(\beta)$. Hence, $g_{\beta,p}''(t) < 0$ for all $t \in [0,r]$ and $p \geq p(\beta)$. This, together with the fact that $g_{\beta,p}'(0) = 0$, implies that $g_{\beta,p}$ is strictly decreasing on $[0,r]$ for all $p \geq p(\beta)$. Moreover, because $g_{\beta,p}(0) = 0$, it follows that $g_{\beta,p}(t) \leq 0$ for all $t \in [0,r]$ and $p \geq p(\beta)$. Hence, $g_{\beta,p}(t) \leq 0$ for all $t \in [0,1]$ and $p \geq p(\beta)$, i.e. $\bw \geq \beta$ for all $p \geq p(\beta)$. This shows that $\bw \rightarrow \log 2$, as $p \rightarrow \infty$.
	
	Next, we show that the sequence $\{\bw\}_{p\geq 2}$ is strictly increasing. Towards this, take any $p\geq 3$. It follows from \cite[Lemma F.1]{mlepaper}, that there exists $a \in (0,1)$, such that $0$ and $a$ are both global maximizers of $g_{\bw,p}$. In particular, $g_{\bw,p}(a) = 0$, and hence, $g_{\bw,p-1}(a) > 0$. The function $\beta \mapsto g_{\beta,p-1}(a)$ being continuous, there exists $\beta < \bw$, such that $g_{\beta,p-1}(a) > 0$. Hence, $\beta_{\mathrm{CW}}^*(p-1) \leq \beta$, establishing that $\beta_{\mathrm{CW}}^*(p-1) < \bw$.
	
	Finally, we show that $\beta_{\mathrm{CW}}^*(2) = 0.5$. By \eqref{conx}, $\sup_{t\in [0,1]} g_{\beta,p}''(t) <0$ for all $\beta < 0.5$. This, coupled with the facts that $g_{\beta,p}'$ and $g_{\beta,p}$ vanish at $0$, implies that $\sup_{t \in [0,1]} g_{\beta,p}(t) = 0$ for all $\beta < 0.5$. Hence, $\beta_{\mathrm{CW}}^*(2) \geq 0.5$. On the other hand, for any $\beta > 0.5$, $g_{\beta,p}''(0) = 2\beta - 1 > 0$ and hence, by continuity of the function $g_{\beta,p}''$ at $0$, there exists $\varepsilon > 0$, such that $\inf_{t\in [0,\varepsilon]} g_{\beta,p}''(t) > 0$. Once again, since $g_{\beta,p}'$ and $g_{\beta,p}$ vanish at $0$,	we have $g_{\beta,p}(t) > 0$ for all $t \in (0,\varepsilon]$. This shows that $\beta_{\mathrm{CW}}^*(2) \leq 0.5$. 
\end{proof}


\begin{thebibliography}{99}





\bibitem{hypergraph_learning}
Sameer Agarwal, Kristin Branson, and Serge Belongie, Higher order learning with graphs, {\it  Proceedings of the 23rd international conference on Machine learning}, 17--24, 2006.




\bibitem{structure_learning}
Animashree Anandkumar, Vincent Y. F. Tan, Furong Huang, and Allan S. Willsky, High-dimensional structure estimation in Ising models: Local separation criterion, {\it The Annals of Statistics}, Vol. 40 (3), 1346--1375  Vol. 2012.




\bibitem{hypergraph_block}
Maria C. Angelini, Francesco Caltagirone, Florent Krzakala, and Lenka Zdeborov\'a, Spectral detection on sparse hypergraphs, {\it 53rd IEEE Annual Allerton Conference},  66--73, 2015.




\bibitem{spatial}
Sudipto Banerjee, Bradley~P. Carlin, and Alan~E. Gelfand,
\newblock Hierarchical modeling and analysis for spatial data,
\newblock {\em Chapman and Hall/CRC, 2014}.



\bibitem{spectral_norm}
Z. D. Bai, Methodologies in spectral analysis of large-dimensional random matrices,
a review (with discussion), {\it Statistica Sinica}, Vol. 9, 611--677, 1999



\bibitem{ab_ferromagnetic_pspin} 
Adriano Barra, Notes on ferromagnetic $p$-spin and REM, {\it Mathematical Methods in the Applied Sciences}, Vol. 32 (7), 783--797, 2009.


\bibitem{besag_lattice}
Julian Besag, Spatial interaction and the statistical analysis of lattice systems, {\it J. Roy. Stat. Soc. B}, Vol. 36, 192--236, 1974.

\bibitem{besag_nl}
Julian Besag, Statistical analysis of non-lattice data, {\it The Statistician}, Vol. 24 (3), 179--195, 1975. 


\bibitem{BM16}
Bhaswar B. Bhattacharya and Sumit Mukherjee, Inference in ising models, {\it Bernoulli}, Vol. 24 (1), 493--525, 2018. 



\bibitem{bovier} 
Anton Bovier, Irina Kurkova, and Matthias L\"owe, Fluctuations of the Free Energy in the REM and the $p$-Spin SK Models, {\it The Annals of Probability}, Vol. 30, 605-651, 2002. 
					
					
						


\bibitem{bresler}
Guy Bresler,  Efficiently learning Ising models on arbitrary graphs, {\it Proceedings Symposium on Theory of Computing (STOC)}, 771--782, 2015.



\bibitem{gb_testing}
Guy Bresler and Dheeraj Nagaraj, Optimal single sample tests for structured versus unstructured
network data, {\it Conference On Learning Theory (COLT)}, 1657--1690, 2018.  


\bibitem{high_tempferro}
Yuan Cao, Matey Neykov and Han Liu, High Temperature Structure Detection in Ferromagnets, {\tt arXiv:1809.08204}, 2019. 


\bibitem{CD16} 
Sourav Chatterjee and Amir Dembo, Nonlinear Large Deviations, {\it Advances in Mathematics}, Vol. 299, 396--450, 2016. 
		
\bibitem{error}
Sourav Chatterjee, An error bound in the Sudakov-Fernique inequality, {\tt arXiv:0510424}, 2008.


\bibitem{chatterjee} 
Sourav Chatterjee, Estimation in spin glasses: A first step,  {\it The Annals of Statistics}, Vol. 35 (5), 1931--1946, 2007. 


		
\bibitem{discrete_tree}
C. Chow and C. Liu, {\it Approximating discrete probability distributions with dependence trees},  {\it IEEE Transactions on Information Theory}, Vol. 14~(3), 462--467, 1968.


\bibitem{comets_exp}
Francis Comets, On consistency of a class of estimators for exponential families of Markov random fields on the lattice,  {\it The Annals of Statistics}, Vol. 20 (1), 455--468, 1992.



\bibitem{comets}
Francis~Comets and Basilis~Gidas,
\newblock Asymptotics of maximum likelihood estimators for the Curie-Weiss
  model,
\newblock {\em The Annals of Statistics}, 19(2):557--578, 1991.


\bibitem{cd_ising_estimation}
Yuval Dagan, Constantinos Daskalakis, Nishant Dikkala, and Anthimos Vardis Kandiros, Estimating Ising models from one sample, {\tt arXiv:2004.09370}, 2020. 



\bibitem{cd_testing}
Constantinos Daskalakis, Nishant Dikkala, and Gautam Kamath, Testing Ising models, {\it IEEE Transactions on Information Theory}, Vol. 65 (11), 6829--6852, 2019.


\bibitem{cd_ising_II}
Constantinos Daskalakis, Nishant Dikkala, and Ioannis Panageas, Logistic regression with peer-group effects via inference in higher-order Ising models, {\it Proceedings of the 23rd International Conference on Artificial Intelligence and Statistics (AISTATS)}, 3653--3663, 2020. 

\bibitem{cd_ising_I}
Constantinos Daskalakis, Nishant Dikkala, and Ioannis Panageas, Regression from dependent observations, {\it Proceedings of the 51st Annual ACM SIGACT Symposium on Theory of Computing (STOC)}, 881--889, 2019. 



\bibitem{cd_trees}
Constantinos Daskalakis, Elchanan Mossel, and Sebastien Roch, Evolutionary trees and the Ising model on the Bethe lattice: A proof of Steel's conjecture, {\it Probability Theory and Related Fields}, Vol. 149 (1), 149--189, 2011.



\bibitem{Gaussian} 
Robert J. Adler and Jonathan E. Taylor, {\it Gaussian Inequalities}, In: Random Fields and Geometry. Springer Monographs in Mathematics. Springer, New York, NY, 2007. 


\bibitem{geman_graffinge}
Stuart Geman and Christine Graffigne, Markov random field image models and their applications
to computer vision, {\it  Proceedings of the International Congress of Mathematicians}, 1496--1517, 1986.



\bibitem{pg_sm}
Promit Ghosal and Sumit Mukherjee, Joint estimation of parameters in Ising model, {\it  Annals of Statistics}, to appear, 2020. 





\bibitem{gidas}
Basilis Gidas, Consistency of maximum likelihood and pseudolikelihood estimators for Gibbs distributions, {\it In Stochastic  Differential  Systems, Stochastic Control Theory and Applications
(W. Fleming and P.-L. Lions, eds.)}, 129--145, Springer, New York, 1988.
 
 

\bibitem{hypergraph_clustering_II}
Debarghya Ghoshdastidar and Ambedkar Dukkipati, Consistency of spectral hypergraph partitioning under planted partition model, {\it Annals of Statistics}, Vol. 45 (1), 289--315, 2017.




 

\bibitem{disease}
Peter~J. Green and Sylvia~Richardson,
\newblock Hidden markov models and disease mapping,
\newblock {\em Journal of the American Statistical Association}, 97:1055--1070,
  2002.

 
\bibitem{guyon}
Xavier Guyon and Hans R. K\"unsch, Asymptotic comparison of estimators in the Ising Model, {\it Stochastic Models, Statistical Methods, and Algorithms in Image Analysis}, Lecture Notes in Statistics, Vol, 74, 177--198, 1992. 



\bibitem{graphical_models_algorithmic}
Linus Hamilton, Frederic Koehler, and Ankur Moitra, Information theoretic properties of Markov Random Fields, and their algorithmic applications, {\it Advances in Neural Information Processing Systems (NIPS)}, 2463--2472, 2017.


\bibitem{multispin_simulations} 
J. R. Heringa, H. W. J. Bl{\"o}te, and A. Hoogland, Phase transitions in self-dual Ising models with multispin interactions and a field, {\it Physical Review Letters}, Vol. 63, 1546--1549, 1989. 





\bibitem{neural}
John J.~Hopfield,
\newblock Neural networks and physical systems with emergent collective
  computational abilities,
\newblock {\em Proc. Natl. Acad. Sci. USA}, 79:2554--2558, 1982.


\bibitem{ising}
Ernst Ising, Beitrag zur theorie der ferromagnetismus, {\it Zeitschrift f\"ur Physik}, Vol. 31, 253--258, 1925. 


\bibitem{ferromagnetic_mean_field} 
Thomas J\"org, Florent Krzakala,  Jorge Kurchan,  A. C. Maggs,  and Jos\'e Pujos, Energy gaps in quantum first-order mean-field--like transitions: The problems that quantum annealing cannot solve, {\it EPL (Europhysics Letters)}, Vol. 89 (4), 40004, 2010. 




\bibitem{ising_partition_approximation}
Jingcheng Liu, Alistair Sinclair, and Piyush Srivastava, The Ising partition function: zeros and deterministic approximation, {\it Journal of Statistical Physics}, Vol. 174, 287--315, 2019. 


\bibitem{hypergraph_image}
Qingshan Liu, Yuchi Huang, and Dimitris N Metaxas. Hypergraph with sampling for image retrieval, {\it Pattern Recognition}, Vol. 44(10-11), 2255--2262, 2011.

\bibitem{tpe}
Erich L. Lehmann and George Casella, {\it Theory of point estimation}, Springer, 2006.



\bibitem{hypergraph_phase_transition}
Thibault Lesieur, L\'eo Miolane, Marc Lelarge, Florent Krzakala, Lenka Zdeborov\'a, Statistical and computational phase transitions in spiked tensor estimation,  {\it IEEE International Symposium on Information Theory (ISIT)}, 511--515, 2017. 



\bibitem{mcdiarmid}
Colin McDiarmid, On the method of bounded differences, \textit{Surveys in Combinatorics}, 148-188, Cambridge University Press, 1989. 
											


\bibitem{spinglass_book} 
Marc M\'ezard, Giorgio Parisi, Miguel A. Virasoro, {\it Spin glass theory and beyond}, World Scientific Lecture Notes in Physics, Vol. 9, World Scientific Publishing Co. Inc., Teane, 1987. 


	
	

\bibitem{sk_optimization}
Andrea Montanari, Optimization of the Sherrington-Kirkpatrick Hamiltonian, {\it IEEE
Symposium on the Foundations of Computer Science} (FOCS), 2019.


			
\bibitem{innovations}
Andrea Montanari and Amin Saberi, The spread of innovations in social networks, {\it Proceedings of the National Academy of Sciences}, Vol. 107 (47), 20196--20201, 2010.



\bibitem{mlepaper}
Somabha Mukherjee, Jaesung Son and Bhaswar B. Bhattacharya, Phase transitions of the maximum likelihood estimates in the $p$-spin Curie-Weiss model, {\tt arXiv:2005.03631}, 2020.

								
\bibitem{ising_testing}
Rajarshi Mukherjee and Gourab Ray, On testing for parameters in Ising model, {\tt arXiv:1906.00456}, 2019. 
									
											
\bibitem{rm_sm}
Rajarshi Mukherjee, Sumit Mukherjee and Ming Yuan, Global testing against sparse alternatives
under Ising models, {\it  Annals of Statistics}, Vol. 46 (5), 2062--2093, 2018.



\bibitem{neykovliu_property}
Matey Neykov and Han Liu, Property testing in high-dimensional Ising models, 
{\it Annals of Statistics}, Vol. 47 (5), 2472--2503, 2019. 



\bibitem{pspinref1}
Masaki~Ohkuwa, Hidetoshi~Nishimori, and Daniel~A. Lidar,
\newblock Reverse annealing for the fully connected $p$-spin model,
\newblock {\em Phys. Rev. A}, 98:022314, 2018.




\bibitem{panchenko_book}
Dmitry Panchenko, {\it The Sherrington-Kirkpatrick model}, Springer, 2013.

	


\bibitem{discrete_mrf_pickard}
David K. Pickard, Inference for discrete Markov Fields: the simplest nontrivial case, {\it Journal of the American Statistical Association}, Vol. 82 (397), 90--96, 1987. 


\bibitem{highdim_ising}
Pradeep Ravikumar, Martin J. Wainwright and John D. Lafferty, High-dimensional Ising model selection using $\ell_1$-regularized logistic regression, {\it The Annals of Statistics}, Vol. 38 (3), 1287--1319, 2010.




\bibitem{multipartite_random_hypergraph}
Vojt\v{e}ch R\H{o}dl,  Andrzej Ruci\'nski, and Mathias Schacht, Ramsey properties of random $k$-partite, $k$-uniform hypergraphs, {\it SIAM Journal on Discrete Mathematics}, Vol. 21 (2), 442--460,  2007. 




\bibitem{graphical_models_binary}
Narayana P. Santhanam and Martin J. Wainwright, Information-theoretic limits of selecting binary
graphical models in high dimensions, {\it IEEE Transactions on Information Theory}, Vol. 58 (7), 4117--4134, 2012. 



\bibitem{ising_suzuki}
Masuo Suzuki, Solution and critical behavior of some  ``Three-Dimensional" Ising Models with a four-spin interaction, {\it Physical Review Letters}, Vol. 28, 507--510, 1972. 



\bibitem{ising_general} 
Masuo Suzuki and Michael E. Fisher,  Zeros of the partition function for the Heisenberg, ferroelectric, and general Ising models, {\it Journal of Mathematical Physics}, Vol. 12 (2), 235--246, 1971.



\bibitem{talagrand_sk}
Micheal Talagrand, The Parisi formula, {\it Annals of Mathematics} (2) Vol. 163, 221--263, 2006. 




\bibitem{talagrand}
Michel~Talagrand,
\newblock {\em Spin Glasses: A Challenge for Mathematicians-Cavity and Mean
  Field Models},
\newblock Springer, Berlin, 2003. 




\bibitem{hypergraph_multimedia}
Shulong Tan, Jiajun Bu, Chun Chen, Bin Xu, Can Wang, and Xiaofei He, Using rich social media information for music recommendation via hypergraph model, {\it ACM Transactions on Multimedia Computing, Communications, and Applications (TOMM)}, Vol. 7(1), Vol. 22, 2011.




\bibitem{hypergraph_gene}
Ze Tian, TaeHyun Hwang, and Rui Kuang. A hypergraph-based learning algorithm for classifying gene expression and arrayCGH data with prior knowledge,  {\it Bioinformatics}, Vol. 25 (21), 2831--2838, 2009.



\bibitem{turban}
Lo\"ic~Turban,
\newblock One-dimensional ising model with multispin interactions,
\newblock {\em Journal of Physics A: Mathematical and Theoretical}, 49(35),  2016.





\bibitem{pspinref2}
Yu~Yamashiro, Masaki~Ohkuwa, Hidetoshi~Nishimori, and Daniel~A. Lidar,
\newblock Dynamics of reverse annealing for the fully-connected $p$-spin model,
\newblock {\em Phys. Rev. A}, 100:052321, 2019. 






\end{thebibliography}
\end{document}